\documentclass[11pt,reqno]{amsart}

\usepackage{amsmath}
\usepackage{amssymb}
\usepackage{amsthm}
\usepackage[dvips]{graphicx}
\usepackage{cases}

\newtheorem{thm}{Theorem}[section]
\newtheorem{defi}[thm]{Definition}
\newtheorem{lem}[thm]{Lemma}
\newtheorem{prop}[thm]{Proposition}
\newtheorem{cor}[thm]{Corollary}
\newtheorem{rem}[thm]{Remark}

\DeclareMathOperator{\Ker}{Ker}
\DeclareMathOperator{\id}{id}
\DeclareMathOperator{\tr}{tr}

\date{}

\begin{document}
\title[Quaternionic $k$-vector fields]{Quaternionic $k$-vector fields on quaternionic K\"{a}hler manifolds}

\author[T. Moriyama]{Takayuki MORIYAMA}
\address[T. Moriyama]{Department of Mathematics, Mie University, Kurimamachiya 1577 Tsu 514-8507, JAPAN}
\email{takayuki@edu.mie-u.ac.jp}

\author[T. Nitta]{Takashi NITTA}
\address[T. Nitta]{Department of Mathematics, Mie University, Kurimamachiya 1577 Tsu 514-8507, JAPAN}
\email{nitta@edu.mie-u.ac.jp}


\maketitle
\thispagestyle{empty}

\begin{abstract}
In this paper, we define a differential operator as a modified Dirac operator. 
Using the operator, we introduce a quaternionic $k$-vector field on a quaternionic K\"{a}hler manifold 
and show that any quaternionic $k$-vector field corresponds to a holomorphic $k$-vector field on the twistor space. 
We calculate the dimension of the space of quaternionic $k$-vector fields on $\mathbb{H}P^n$. 
\end{abstract}

\section{Introduction}
Let $(M, g)$ be a quaternionic K\"{a}hler manifold, that is, a $4n$-dimensional Riemannian manifold whose holonomy group is reduced to a subgroup of $\textrm{Sp}(n)\cdot \textrm{Sp}(1)$. 
Let $E$ and $H$ denote the associated bundles with the canonical representations of $\textrm{Sp}(n)$ and $\textrm{Sp}(1)$ on $\mathbb{C}^{2n}$ and $\mathbb{C}^2$, respectively. 
Then $TM\otimes \mathbb{C} =E\otimes_{\mathbb{C}} H$. 
The complex vector bundles $E, H$ admit symplectic structures and anti-$\mathbb{C}$-linear maps which square to $-\id $. 
The Levi-Civita connection of $(M,g)$ is decomposed into connections of $E$ and $H$. 
They induce the covariant derivative $\nabla :\Gamma(\wedge^k E\otimes S^kH) \to \Gamma(\wedge^k E\otimes  S^kH \otimes E^*\otimes H^*)$ 
where $\Gamma(\wedge^k E\otimes S^kH)$ means the space of smooth sections of $\wedge^k E\otimes S^kH$. 
The bundle $\wedge^k E\otimes S^kH \otimes E^*\otimes H^*$ is isomorphic to $\wedge^k E\otimes  E^* \otimes S^kH\otimes H$ 
by the symplectic structure of $H$. 
Moreover, $S^kH\otimes H\cong S^{k+1}H\oplus S^{k-1}H$ by the Clebsch-Gordan decomposition. 
Thus, the covariant derivative $\nabla$ is regarded as 
\[
\nabla : \Gamma(\wedge^k E\otimes S^kH) \to \Gamma(\wedge^k E\otimes E^* \otimes S^{k+1}H) \oplus \Gamma(\wedge^k E\otimes E^*\otimes S^{k-1}H). 
\]
Dirac operator $\mathfrak{D}_{\wedge^k E}$ is defined by the $\wedge^k E\otimes E^*\otimes S^{k+1}H$-part of $\nabla$ (c.f.~\cite{B}) : 
\[
\mathfrak{D}_{\wedge^k E}: \Gamma(\wedge^k E\otimes S^kH) \to \Gamma(\wedge^k E\otimes E^*\otimes S^{k+1}H)
\]
The trace of $(\otimes^k E) \otimes E^*$ induces the map $\tr : \wedge^k E \otimes E^*\to \wedge^{k-1} E$ 
by the restriction of $\otimes^k E\otimes E^*$ to $\wedge^k E\otimes E^*$. 
Let $(\wedge^k E\otimes E^*)_0$ denote the kernel of $\tr: \wedge^k E \otimes E^*\to \wedge^{k-1} E$. 
The bundle $\wedge^k E\otimes E^*$ is decomposed into $(\wedge^k E\otimes E^*)_0$ and $(\wedge^{k-1} E)\wedge \id_E$. 
We define an operator 
\[
\mathfrak{D}_{\wedge^k E}^0: \Gamma(\wedge^k E\otimes S^kH) \to \Gamma((\wedge^k E\otimes E^*)_0\otimes S^{k+1}H)
\]
as the $(\wedge^k E\otimes E^*)_0$-part of $\mathfrak{D}_{\wedge^k E}$. 
A section of $\wedge^k E\otimes S^kH$ is a $k$-vector field 
since $\wedge^k E\otimes S^kH$ is the subbundle of $\wedge^k TM\otimes \mathbb{C}$. 
\begin{defi}
{\rm 
A section $X$ of $\wedge^k E\otimes S^kH$ is a \textit{quaternionic $k$-vector field} on $M$ 
if $\mathfrak{D}_{\wedge^k E}^0(X)=0$ for $1\le k\le 2n-1$ and 
$\mathfrak{D}_{\wedge^{2n-1} E}\circ \tr\circ \mathfrak{D}_{\wedge^{2n} E}(X)=0$ for $k=2n$. 
}
\end{defi}
The twistor space $Z$ of $M$ is given by the projective bundle of the dual bundle $H^*$. 
It has a natural complex structure and a real structure. 
It is known that a section $X$ of $\wedge^k E\otimes S^kH$ with $\mathfrak{D}_{\wedge^k E}(X)=0$ is lifted to a holomorphic $k$-vector field on $Z$. 
However, any holomorphic $k$-vector field on $Z$ does not correspond to such a $k$-vector field on $M$. 
The following is the main theorem : 
\begin{thm}\label{s1t1}
Let $\mathcal{Q}(\wedge^k E\otimes S^kH)$ be the sheaf of quaternionic $k$-vector fields on $M$ and 
$\mathcal{O}(\wedge^k T^{1,0}Z)$ that of holomorphic $k$-vector fields on $Z$. 
Then, $H^0(\mathcal{Q}(\wedge^k E\otimes S^kH))$ is isomorphic to $H^0(\mathcal{O}(\wedge^k T^{1,0}Z))$. 
\end{thm}
The space $\Gamma(\wedge^k E\otimes S^kH)$ admits a real structure $\tau$. 
Any $\tau$-invariant section of $\wedge^k E\otimes S^kH$ is a real $k$-vector field on $M$. 
On the other hand, the real structure of $Z$ induces the real structure $\widehat{\tau}$ on the space of holomorphic $k$-vector fields. 
We obtain the following theorem : 
\begin{thm}\label{s1t2}
The space $H^0(\mathcal{Q}(\wedge^k E\otimes S^kH))^{\tau}$ of quaternionic real $k$-vector fields on $M$ is isomorphic to 
the space $H^0(\mathcal{O}(\wedge^k T^{1,0}Z))^{\widehat{\tau}}$ of $\widehat{\tau}$-invariant holomorphic $k$-vector fields on $Z$. 
\end{thm}

This paper is organized as follows. 
In Section 2, we prepare some fundamental facts about quaternionic vector spaces and quaternionic K\"{a}hler manifolds. 
In Section 3 and 4, we see some theorems in twistor theory from the view point of the principal bundle of $H^*$ and the connection. 
In Section 5, we introduce quaternionic sections on quaternionic K\"{a}hler manifolds 
as sections of $\wedge^kE\otimes S^mH$ satisfying some differential equations. 
In the final section, we provide the definition of quaternionic $k$-vector fields. 
We show Theorem~\ref{s1t1} and Theorem~\ref{s1t2} (see Theorem~\ref{s5.8t1} and Theorem~\ref{s5.9t1}). 
As an example, we consider the $n$-dimensional quaternionic projective space $\mathbb{H}P^n$ 
and compute the dimension of the space of quaternionic $k$-vector fields. 

\section{Preliminaries}
\subsection{The $n$-dimensional quaternionic vector space $\mathbb{H}^n$}\label{s2.1} 
Let $\mathbb{H}$ be Hamilton's quaternionic number field $\mathbb{R}+i\mathbb{R}+j\mathbb{R}+k\mathbb{R}$. 
We define endomorphisms $i_-, j_-, k_-$ of $\mathbb{H}^n$ as the right action of $i, j, k$ on $\mathbb{H}^n$, respectively. 
The product $\textrm{Sp}(n)\times \textrm{Sp}(1)$ acts on $\mathbb{H}^n$ by 
$(A,q)\cdot \xi=A\xi \bar{q}$ for $(A,q)\in \textrm{Sp}(n)\times \textrm{Sp}(1)$ and $\xi \in \mathbb{H}^n$. 
Since $(-1,-1)$ acts as the identity, 
the quotient $\textrm{Sp}(n)\cdot \textrm{Sp}(1)=\textrm{Sp}(n)\times \textrm{Sp}(1)/\mathbb{Z}_2$ acts on $\mathbb{H}^n$. 
The action of $\textrm{Sp}(n)\cdot \textrm{Sp}(1)$ on $\mathbb{H}^n$ preserves 
the subspace $\left\langle i_-, j_-, k_- \right\rangle$ of $\textrm{End}_{\mathbb{R}}(\mathbb{H}^{n})$. 
Under the identification $\mathbb{H}^n\cong \mathbb{R}^{4n}$, 
we denote by $I_-, J_-, K_-$ the endomorphisms of $\mathbb{R}^{4n}$ corresponding to $i_-, j_-, k_-$, respectively. 
Then $\textrm{Sp}(n)\cdot \textrm{Sp}(1)$ is identified with a subgroup of $\textrm{SO}(4n)$ which preserves $\left\langle I_-, J_-, K_- \right\rangle$. 
We denote by $Q_0$ the subspace $\left\langle I_-, J_-, K_- \right\rangle$. 
In particular, $Sp(n)$ is considered as a subgroup of $SO(4n)$ preserving each $I_-, J_-, K_-$.

\subsection{Quaternionic K\"{a}hler manifolds}\label{s2.5} 
Let $(M,g)$ be a Riemannian manifold of dimension $4n$. 
A subbundle $Q$ of $\textrm{End}(TM)$ is called an \textit{almost quaternionic structure} 
if there exists a local basis $I, J, K$ of $Q$ such that $I^2=J^2=K^2=-\id$ and $K=IJ$. 
A pair $(Q, g)$ is an \textit{almost quaternionic Hermitian structure} 
if any section $\varphi$ of $Q$ satisfies $g(\varphi X, Y)+g(X,\varphi Y)=0$ for $X, Y\in TM$. 
For $n\ge 2$, if the Levi-Civita connection $\nabla$ preserves $Q$, 
then $(Q, g)$ is called a \textit{quaternionic K\"{a}hler structure}, and $(M, Q, g)$ a \textit{quaternionic K\"{a}hler manifold}. 
A Riemannian manifold is a quaternionic K\"{a}hler manifold if and only if 
the holonomy group is reduced to a subgroup of $\textrm{Sp}(n)\cdot \textrm{Sp}(1)$. 
Alekseevskii~\cite{A} shows that a quaternionic K\"{a}hler manifold is Einstein and the curvature of $Q$ is described by the scalar curvature (we also refer to~\cite{I, S1}). 
For $n=1$, since $\textrm{Sp}(1)\cdot \textrm{Sp}(1)$ is $SO(4)$, 
a manifold satisfying the above condition is just an oriented Riemannian manifold. 
A $4$-dimensional oriented Riemannian manifold $M$ is said to be a \textit{quaternionic K\"{a}hler manifold} 
if it is Einstein and self-dual. 
The scalar curvature of a quaternionic K\"{a}hler manifold $M$ vanishes if and only if 
the holonomy group is reduced to $\textrm{Sp}(n)$, that is, $M$ is a hyperk\"ahler manifold. 

The frame bundle of a quaternionic K\"{a}hler manifold $M$ is reduced to a principal $\textrm{Sp}(n)\cdot \textrm{Sp}(1)$-bundle $F$. 
The Levi-Civita connection $\nabla$ induces a connection of $F$. 
The bundle $F$ is lifted to a principal $\textrm{Sp}(n)\times \textrm{Sp}(1)$-bundle $\widetilde{F}$, locally. 
Given a representation $V$ of $\textrm{Sp}(n)\times \textrm{Sp}(1)$, 
we construct an associated bundle $\widetilde{F}\times _{\textrm{Sp}(n)\times \textrm{Sp}(1)} V$ with the induced connection. 
The bundle is globally defined on $M$ if either $\widetilde{F}$ exists globally 
or $(-1,-1)\in \textrm{Sp}(n)\times \textrm{Sp}(1)$ acts as the identity in the representation. 
The obstruction to the global existence of $\widetilde{F}$ is provided by a cohomology class $\varepsilon\in H^2(M,\mathbb{Z}_2)$ introduced by Marchiafava-Romani~\cite{MR}. 
In $n$ is odd, the class $\varepsilon$ is given by the second Stiefel-Whitney class $w_2$. 

The symplectic group $\textrm{Sp}(n)$ acts on the right $\mathbb{H}$-module $\mathbb{H}^n$ by 
$A\xi$ for $A\in \textrm{Sp}(n)$ and $\xi \in \mathbb{H}^n$. 
On the other hand, $\textrm{Sp}(1)$ has an action on the left $\mathbb{H}$-module $\mathbb{H}$ by 
$\xi \bar{q}$ for $q\in \textrm{Sp}(1)$ and $\xi \in \mathbb{H}$. 
Let $E$, $H$ denote the associated bundles with the representations $\textrm{Sp}(n)$, $\textrm{Sp}(1)$ on $\mathbb{H}^n$, $\mathbb{H}$, respectively. 
Then $E$ is the right $\mathbb{H}$-module bundle and $H$ is the left $\mathbb{H}$-module bundle. 
The dual representations of $\textrm{Sp}(n)$ and $\textrm{Sp}(1)$ induce the left $\mathbb{H}$-module bundle $E^*$ and the right $\mathbb{H}$-module bundle $H^*$. 
Then 
\[
TM=E\otimes_{\mathbb{H}} H, \quad T^*M=H^*\otimes_{\mathbb{H}} E^*.
\]
The $\mathbb{H}$-bundles $E, H$ are regarded as the $\mathbb{C}$-vector bundles with anti $\mathbb{C}$-linear maps $J_E$, $J_H$ satisfying $J_E^2=-\id_E, J_H^2=-\id_H$. 
Then there exist symplectic structures $\omega_E$, $\omega_H$ on $E$, $H$ 
which are compatible with $J_E, J_H$, respectively. 
As the same manner, $E^*$ and $H^*$ are $\mathbb{C}$-vector bundles with anti $\mathbb{C}$-linear maps $J_{E^*}$, $J_{H^*}$. 
Then $J_E^*(\alpha)=-\overline{\alpha \circ J_E}$, $J_H^*(\beta)=-\overline{\beta \circ J_H}$ for $\alpha \in E^*, \beta \in H^*$. 
The correspondences $\xi\mapsto \omega_E(\cdot ,\xi)$, $u\mapsto \omega_H(\cdot ,u)$ provide 
the $\mathbb{C}$-isomorphisms $E \cong E^*$, $H \cong H^*$, which are denoted by $\omega_E^{\sharp}$, $\omega_H^{\sharp}$. 
The tensor space $(\otimes^pE)\otimes (\otimes^q H)\otimes (\otimes^{q'} H^*)\otimes (\otimes^{p'} E^*)$ is globally defined if $p+p'+q+q'$ is even. 
Then $J_E^p \otimes J_H^q\otimes J_{H^*}^{q'} \otimes J_{E^*}^{p'}$ is a real structure on the tensor space. 
In particular, $J_E\otimes J_H$ and $J_{H^*}\otimes J_{E^*}$ are real structures on $E\otimes_{\mathbb{C}} H$ and $H^*\otimes_{\mathbb{C}} E^*$. 
The real forms of $J_E\otimes J_H$ and $J_{H^*}\otimes J_{E^*}$ are $TM$ and $T^*M$, respectively. 
Hence 
\[
TM\otimes \mathbb{C} =E\otimes_{\mathbb{C}} H, \quad T^*M\otimes \mathbb{C}=H^*\otimes_{\mathbb{C}} E^*.
\]
The tensor product $\omega_E\otimes \omega_H$ is the complexification of the Riemannian metric $g$. 
The technique is called \textit{$EH$-formalism}, which was introduced by Salamon in~\cite{S1}.

\subsection{The twistor space}\label{s2.6}
The quaternionic structure $Q$ is an associated bundle with the representation of the action of $\textrm{Sp}(n)\cdot \textrm{Sp}(1)$ on $Q_0$. 
The representation is reduced to that of $\textrm{Sp}(1)$ on the subspace $\left\langle i, j, k \right\rangle$ of $\textrm{End}_{\mathbb{H}}(\mathbb{H})$, 
where $i, j, k$ act on a left $\mathbb{H}$-module $\mathbb{H}$ by the right multiplication. 
Hence, $Q$ is considered as a subbundle of the real vector bundle $\textrm{End}_{\mathbb{H}}(H)$. 
We identify $\textrm{End}_{\mathbb{H}}(H)$ with the real form of $\textrm{End}_{\mathbb{C}}(H)=H\otimes_{\mathbb{C}}H^*$ with respect to $J_H\otimes J_{H^*}$. 
Then $Q$ is contained in $\textrm{End}_{\mathbb{C}}(H)$. 
Let $u$ be an $\mathbb{H}$-frame of $H$. 
We define local sections $I, J, K$ of $\textrm{End}_{\mathbb{H}}(H)$ as 
$I(hu)=hiu,\ J(hu)=hju,\ K(hu)=hku$ for any $h\in \mathbb{H}$. 
Then $\{I, J, K\}$ is a local basis of $Q$ and represented by elements
\begin{equation}\label{s2eq3}
I=i(u\otimes u^*-ju\otimes (ju)^*),\quad J=ju\otimes u^*-u\otimes (ju)^*,\quad K=i(ju\otimes u^*+u\otimes (ju)^*)
\end{equation}
of $\textrm{End}_{\mathbb{C}}(H)$ for the $\mathbb{C}$-frame $\{u, ju\}$ of $H$. 
Let $Z$ be a sphere bundle 
\[
Z=\{aI+bJ+cK \in Q \mid a^2+b^2+c^2=1 \}
\]
over $M$. Let $f:Z \to M$ denote the projection. 
The bundle $Z$ is called a \textit{twistor space} of the quaternionic K\"{a}hler manifold $M$. 
Let $I'$ be an element of $Z$. 
We set $x$ as the point $f(I')$ of $M$. 
There exists an element $(u')^*$ of $H^*_x$ such that $I'=i(u'\otimes (u')^*-ju'\otimes (ju')^*)$. 
The element $I'$ is extended to a complex structure of $T_xM$ by the decomposition $TM\otimes \mathbb{C}=E\otimes H$. 
It follows from $(ju')^*=-(u')^*j$ that 
\begin{equation*}\label{s2eq5}
\wedge ^{1,0}T_x^*M=E_x^*\otimes \left< (u')^*\right>_{\mathbb{C}},\quad \wedge ^{0,1}T_x^*M=E_x^*\otimes \left< (u')^*j \right>_{\mathbb{C}}. 
\end{equation*}
Thus any element of $Z$ defines a complex structure of the tangent space of $M$. 
Let $p: P(H^*)\to M$ be a frame bundle of $H^*$, whose fiber consists of right $\mathbb{H}$-bases of $H^*$. 
Then $P(H^*)$ is a principal $\textrm{GL}(1,\mathbb{H})$-bundle by the right action. 
The twistor space $Z$ is regarded as the quotient space $P(H^*)/\textrm{GL}(1,\mathbb{C})$. 
We denote by $\pi: P(H^*)\to Z$ the quotient map and regard $P(H^*)$ as a principal $\textrm{GL}(1,\mathbb{C})$-bundle over $Z$. 
By the definition, the twistor space $Z$ is a $\mathbb{C}P^1$-bundle over $M$.

\section{The principal bundle $P(H^*)$}
Let $(M,g)$ be a quaternionic K\"{a}hler manifold of dimension $4n$. 
The frame bundle of $H^*$ is the principal $\textrm{GL}(1,\mathbb{H})$-bundle $p:P(H^*)\to M$. 
Let $\mathcal{A}^q(\wedge^kE\otimes S^mH)$ denote the sheaf of $\wedge^kE\otimes S^mH$-valued smooth $q$-forms on $M$. 

\subsection{Lift of $\mathcal{A}^q(\wedge^kE\otimes S^mH)$ to $P(H^*)$}
In this section, the bundles $H$ and $H^*$ are regarded as bundles of the left $\mathbb{C}$-module and the right $\mathbb{C}$-module, respectively. 
We denote the complex representation $\rho$ of $\textrm{GL}(1,\mathbb{H})$ on $\mathbb{H}$ 
by $\rho(a)h=ah$ for $a\in \textrm{GL}(1,\mathbb{H})$ and $h\in \mathbb{H}$. 
The dual representation $\rho^*$ of $\rho$ is given by $\rho^*(a)h=h a^{-1}$ for $a\in \textrm{GL}(1,\mathbb{H})$ and $h\in \mathbb{H}$. 
Then $H$ is the associated bundle $P(H^*)\times_{\rho^*}\mathbb{H}$ with the representation $\rho^*$. 
Let $S^m\mathbb{H}$ be the $\mathbb{C}$-vector space $S^m\mathbb{C}^{2}$. 
We denote by $s^m\rho^*$ the representation of $\textrm{GL}(1,\mathbb{H})$ on the $m$-th symmetric tensor $S^m\mathbb{H}$ of the $\mathbb{C}$-vector space $\mathbb{H}$ induced by $\rho^*$. 
The bundle $S^mH$ over $M$ is given by the associated bundle $S^mH=P(H^*)\times_{s^m\rho^*}S^m\mathbb{H}$ with the representation $s^m\rho^*$. 
In the case $m=0$, $\rho_0^*=\id_{\mathbb{C}}$ and $S^0H=\underline{\mathbb{C}}$. 

The point $u^*\in P(H^*)$ induces the point $u$ of $P(H)$ by taking the dual $\mathbb{H}$-basis of $H$ at $p(u^*)$. 
The $\mathbb{H}$-basis $u$ provides the $\mathbb{C}$-basis $\{u, ju\}$ of the $\mathbb{C}$-vector bundle $H$. 
Thus, any element $u$ of $P(H)$ is regarded as a $\mathbb{C}$-isomorphism $u : \mathbb{H} \to H_{p(u)}$, 
and it induces a $\mathbb{C}$-isomorphism $S^m\mathbb{H} \to S^mH_{p(u)}$, which is denoted by $u$ for simplicity. 
For a section $\xi$ of $S^mH\to M$, a section $\widetilde{\xi}$ of the trivial bundle $S^m\mathbb{H} \to P(H^*)$ is given by 
$\widetilde{\xi}_{u^*}=u^{-1}\xi(p(u^*))$ at $u^*\in P(H^*)$. 
Then $(R_a)^*\widetilde{\xi}=(s^m\rho^*)(a^{-1})\widetilde{\xi}$ for $a\in \textrm{GL}(1,\mathbb{H})$. 
We write $\rho_m(a)$ as $\rho_m(a)=(s^m\rho^*)(a^{-1})$ for $a\in \textrm{GL}(1,\mathbb{H})$. 
Conversely, any section $\xi$ of $S^mH\to M$ is induced by 
a section $\widetilde{\xi}$ of $S^m\mathbb{H}\to P(H^*)$ such that $(R_a)^*\widetilde{\xi}=\rho_m(a)\widetilde{\xi}$ for any $a\in \textrm{GL}(1,\mathbb{H})$. 

Let $\mathcal{A}^q(S^m\mathbb{H})$ and $\mathcal{A}^q_{P(H^*)}(S^m\mathbb{H})$ be the sheaf of $S^m\mathbb{H}$-valued smooth $q$-forms on $M$ and $P(H^*)$, respectively. 
We consider the pull-back $p^*\mathcal{A}^q(S^m\mathbb{H})$ as the subsheaf of $\mathcal{A}^q_{P(H^*)}(S^m\mathbb{H})$. 
If $\widetilde{\xi}$ is an element of the inverse image $p^{-1}p_*(p^*\mathcal{A}^q(S^m\mathbb{H}))$ of the direct image of $p^*\mathcal{A}^q(S^m\mathbb{H})$, 
then $(R_a)^*\widetilde{\xi}$ is in $p^{-1}p_*(p^*\mathcal{A}^q(S^m\mathbb{H}))$ for any $a\in \textrm{GL}(1,\mathbb{H})$. 
We define a sheaf $\widetilde{\mathcal{A}}^q_m(S^m\mathbb{H})$ by 
\[
\widetilde{\mathcal{A}}^q_m(S^m\mathbb{H})=\{\widetilde{\xi}\in p^{-1}p_*(p^*\mathcal{A}^q(S^m\mathbb{H})) \mid (R_a)^*\widetilde{\xi}=\rho_m(a)\widetilde{\xi},\ \forall a\in \textrm{GL}(1,\mathbb{H})\}.
\]
The sheaf $\widetilde{\mathcal{A}}^q_0(S^0\mathbb{H})$ is just that of pull-back of smooth $q$-forms on $M$ by $p$. 
We denote $\widetilde{\mathcal{A}}^q_0(S^0\mathbb{H})$ by $\widetilde{\mathcal{A}}^q$ for simplicity. 
In particular, $\widetilde{\mathcal{A}}^0$ is the sheaf of smooth functions on $P(H^*)$ which are constant along each fiber of $p$. 
Then 
\[
\widetilde{\mathcal{A}}^q_m(S^m\mathbb{H})=\widetilde{\mathcal{A}}^q\otimes_{\widetilde{\mathcal{A}}^0} \widetilde{\mathcal{A}}^0_m(S^m\mathbb{H}).
\] 
An element $\xi \in \mathcal{A}^q(S^mH)$ induces $\widetilde{\xi}\in \widetilde{\mathcal{A}}^q_m(S^m\mathbb{H})$ by 
\[
\widetilde{\xi}_{u^*}=u^{-1}(p^*\xi)_{u^*}
\]
at each point $u^*\in P(H^*)$. 
The Levi-Civita connection of $M$ induces a connection $\nabla$ on $H$, 
and the covariant exterior derivative $d^{\nabla}: \mathcal{A}^q(S^mH)\to \mathcal{A}^{q+1}(S^mH)$. 
The tangent space $T_{u^*}P(H^*)$ is decomposed into the horizontal space $\widetilde{\mathcal{H}}_{u^*}$ and 
the vertical space $\widetilde{\mathcal{V}}_{u^*}$ at $u^*\in P(H^*)$. 
We define $d_{\widetilde{\mathcal{H}}}: \widetilde{\mathcal{A}}^q_m(S^m\mathbb{H})\to \widetilde{\mathcal{A}}^{q+1}_m(S^m\mathbb{H})$ 
by the exterior derivative restricted to the horizontal $\widetilde{\mathcal{H}}$. 
The sheaf $\mathcal{A}^q(S^mH)$ is isomorphic to $\widetilde{\mathcal{A}}^q_m(S^m\mathbb{H})$ by the correspondence $\xi \mapsto \widetilde{\xi}$. 
Moreover, $\widetilde{d^{\nabla}\xi}=d_{\widetilde{\mathcal{H}}}\widetilde{\xi}$ for any $\xi\in \mathcal{A}^q(S^mH)$ (c.f. Chapter II, \S 5 in \cite{KN1}). 

We fix a point $u_0^*$ of $P(H^*)$. 
The complex coordinate $(z,w)$ of the fiber is given by $u_0^*(z+jw)$. 
A function $f$ on $P(H^*)$ is \textit{a polynomial of degree $(m-i,i)$ along fiber} 
if $f(u_0^*(z+jw))$ is a polynomial of $z,w,\bar{z},\bar{w}$ of degree $m$ such that 
$(R_c)^*f=c^{m-i}\bar{c}^{i}f$ for $c\in \textrm{GL}(1,\mathbb{C})$. 
We denote by $\widetilde{\mathcal{A}}^0_{(m-i,i)}$ the sheaf of elements of $p^{-1}p_*\mathcal{A}^0_{P(H^*)}(\mathbb{C})$ which are polynomials of degree $(m-i,i)$ along fiber on $P(H^*)$. 
We also define a sheaf $\widetilde{\mathcal{A}}^q_{(m-i,i)}$ as 
\[
\widetilde{\mathcal{A}}^q_{(m-i,i)}=\widetilde{\mathcal{A}}^q \otimes_{\widetilde{\mathcal{A}}^0} \widetilde{\mathcal{A}}^0_{(m-i,i)}. 
\]

Let $a_1 a_2 \cdots a_m$ denote the symmetrization $\frac{1}{m!}\sum_{\sigma\in S_m}a_{\sigma(1)}\otimes \cdots \otimes a_{\sigma(m)}$ 
of $a_1\otimes \cdots \otimes a_m  \in \otimes^m \mathbb{H}$ where $S_m$ is the symmetric group of degree $m$. 
The set $\{1^m, 1^{m-1} j, 1^{m-2} j^2,\dots, j^m \}$ is a $\mathbb{C}$-basis of the vector space $S^m\mathbb{H}$. 
An $S^m\mathbb{H}$-valued $q$-form $\widetilde{\xi}$ on $P(H^*)$ is given by 
\[
\widetilde{\xi}=\widetilde{\xi}_0 1^m+\widetilde{\xi}_1 1^{m-1} j+\widetilde{\xi}_2 1^{m-2} j^2+ \dots+ \widetilde{\xi}_m j^m
\]
for $q$-forms $\widetilde{\xi}_0,\dots,\widetilde{\xi}_m$ on $P(H^*)$. 
\begin{prop}\label{s3.1p2} 
If $\widetilde{\xi}$ is in $\widetilde{\mathcal{A}}^q_m(S^m\mathbb{H})$, then $\widetilde{\xi}_0 \in \widetilde{\mathcal{A}}^q_{(m,0)}$. 
Conversely, for $\widetilde{\xi}_0 \in \widetilde{\mathcal{A}}^q_{(m,0)}$, 
there exists a unique section $\widetilde{\xi} \in \widetilde{\mathcal{A}}^q_m(S^m\mathbb{H})$ in which the coefficient of $1^m$ is $\widetilde{\xi}_0$. 
\end{prop}
\begin{proof} 
It suffices to show the case $q=0$. 
Let $\widetilde{\xi}$ be an $S^m\mathbb{H}$-valued function on $P(H^*)$. 
We fix a point $u_0^*$ of $P(H^*)$. 
The complex coordinate $(z,w)$ of the fiber is given by $u_0^*(z+jw)$. 
If we take $a=z+jw$, then there exists 
\[
\left(
\begin{array}{cccc}
(1a)^m \\
(1a)^{m-1}ja\\
\vdots \\
(ja)^m
\end{array}
\right)
=
\left(
\begin{array}{cccc}
p_{00} &p_{01} & \cdots&p_{0m}  \\
p_{10} &p_{11} & \cdots& p_{1m} \\
\vdots &\vdots &\ddots&\vdots  \\
p_{m0} &p_{1m}&\cdots &p_{mm} 
\end{array}
\right)
\left(
\begin{array}{cccc}
1^m \\
1^{m-1}j \\
\vdots \\
j^m
\end{array}
\right)
\]
where $p_{ij}$ is a polynomial of $z,w,\bar{z},\bar{w}$ of degree $m$. 
Each component $p_{ii'}$ is a polynomial of degree $(m-i',i')$ for $i,i'=0,1,\dots,m$. 
In particular, $p_{i0}=(-1)^iz^{m-i}w^i$ for $i=0,1,\dots,m$. 
It follows from $(R_a)^*\widetilde{\xi}=\rho_m(a)\widetilde{\xi}$ that 
$\widetilde{\xi}_0(u_0^*a)=\sum_{i=0}^m \widetilde{\xi}_i(u_0^*)p_{i0}=\sum_{i=0}^m (-1)^i\widetilde{\xi}_i(u_0^*)z^{m-i}w^i$. 
Conversely, $\widetilde{\xi}_0(u_0^*a)=\sum_{i=0}^m c_{i}z^{m-i}w^i$ for complex number $c_0,c_1,\dots, c_m$. 
We define an $S^m\mathbb{H}$-valued function $\widetilde{\xi}$ on $P(H^*)$ as 
$\widetilde{\xi}_{u_0^*a}=\sum_{i=0}^{m}(-1)^i c_i a^{m-i} (ja)^i$ for any $a\in \textrm{GL}(1,\mathbb{H})$. 
Then $\widetilde{\xi}$ is in $\widetilde{\mathcal{A}}^q_m(S^m\mathbb{H})$. 
By the definition, the coefficient of $1^m$ in $\widetilde{\xi}_{u_0^*a}$ is $\sum_{i=0}^m (-1)^i c_ip_{i0}=\sum_{i=0}^m c_iz^{m-i}w^i=\widetilde{\xi}_0$, 
and hence it completes the proof
\end{proof}
The coefficient $\widetilde{\xi}_{i}$ of $\widetilde{\xi}$ is in $\widetilde{\mathcal{A}}^q_{(m-i,i)}$. 
Proposition~\ref{s3.1p2} implies the following  
\begin{cor}\label{s3.1c1}
The sheaf $\mathcal{A}^q(S^mH)$ is isomorphic to $\widetilde{\mathcal{A}}^q_{(m,0)}$ 
by the correspondence $\xi \mapsto \widetilde{\xi}_0$. 
Moreover, $(\widetilde{d^{\nabla}\xi})_0=d_{\widetilde{\mathcal{H}}}\widetilde{\xi}_0$ for any $\xi\in \mathcal{A}^q(S^mH)$. $\hfill\Box$
\end{cor}

We extend the above argument of $\mathcal{A}^q(S^mH)$ to $\mathcal{A}^q(\wedge^k E\otimes S^mH)$ as follows. 
Let $\widetilde{\mathcal{A}}^q(\wedge^k E)$ denote the sheaf of pull-back of $\wedge^k E$-valued smooth $q$-forms on $M$ by $p$. 
We define $\widetilde{\mathcal{A}}^q_m(\wedge^k E\otimes S^m\mathbb{H})$ and $\widetilde{\mathcal{A}}^q_{(m-i,i)}(\wedge^k E)$ as the sheaves 
\[
\widetilde{\mathcal{A}}^q_m(\wedge^k E\otimes S^m\mathbb{H})=\widetilde{\mathcal{A}}^q(\wedge^k E)\otimes_{\widetilde{\mathcal{A}}^0} \widetilde{\mathcal{A}}^0_m(S^m\mathbb{H}) 
\]
and 
\begin{equation*}\label{s3.1eq5}
\widetilde{\mathcal{A}}^q_{(m-i,i)}(\wedge^k E)=\widetilde{\mathcal{A}}^q(\wedge^k E)\otimes_{\widetilde{\mathcal{A}}^0} \widetilde{\mathcal{A}}^0_{(m-i,i)}. 
\end{equation*}
Any element $\widetilde{\xi}$ of $\widetilde{\mathcal{A}}^q_m(\wedge^k E\otimes S^m\mathbb{H})$ is written by 
\begin{equation}\label{s3.1eq6}
\widetilde{\xi}=\widetilde{\xi}_0 1^m+\widetilde{\xi}_1 1^{m-1} j+\widetilde{\xi}_2 1^{m-2} j^2+ \dots+ \widetilde{\xi}_m j^m
\end{equation}
for $p^{-1}(\wedge^k E)$-valued 1-forms $\widetilde{\xi}_0,\dots,\widetilde{\xi}_m$. 
The Levi-Civita connection induces connections of $E$, $H$ and the covariant exterior derivative $d^{\nabla}: \mathcal{A}^q(\wedge^k E\otimes S^mH)\to \mathcal{A}^{k+1}(\wedge^k E\otimes S^mH)$. 
Let $\widetilde{\mathcal{H}}$ be the horizontal subbundle of $TP(H^*)$. 
By the same argument in Proposition \ref{s3.1p2} and Corollary \ref{s3.1c1}, we obtain  
\begin{cor} \label{s3.1c2}
(i) $\widetilde{\mathcal{A}}^q_m(\wedge^k E\otimes S^m\mathbb{H})\cong \widetilde{\mathcal{A}}^q_{(m,0)}(\wedge^k E)$ by $\widetilde{\xi} \mapsto \widetilde{\xi}_0$. 
Moreover, $(d_{\widetilde{\mathcal{H}}}\widetilde{\xi})_0=d_{\widetilde{\mathcal{H}}}\widetilde{\xi}_0$ for any $\widetilde{\xi}\in \widetilde{\mathcal{A}}^q_m(\wedge^k E\otimes S^m\mathbb{H})$. \\
(ii) $\mathcal{A}^q(\wedge^k E\otimes S^mH)\cong \widetilde{\mathcal{A}}^q_{(m,0)}(\wedge^k E)$ by $\xi \mapsto \widetilde{\xi}_0$. 
Moreover, $(\widetilde{d^{\nabla}\xi})_0=d_{\widetilde{\mathcal{H}}}\widetilde{\xi}_0$ for any $\xi\in \mathcal{A}^q(\wedge^k E\otimes S^mH)$. $\hfill\Box$ 
\end{cor}
For $\xi\in \mathcal{A}^q(\wedge^k E\otimes S^mH)$, 
the element $\widetilde{\xi}_0\in \widetilde{\mathcal{A}}^q_{(m,0)}(\wedge^k E)$ is said to be a \textit{lift to $P(H^*)$}.

\subsection{Real structures on $P(H^*)$}\label{s3.2} 
We define an anti-$\mathbb{C}$-linear map $\tau : \mathcal{A}^q(\wedge^kE\otimes S^mH) \to \mathcal{A}^q(\wedge^kE\otimes S^mH)$ 
by 
\begin{equation*}\label{s3.1.5eq0}
\tau(\xi)=\sum_i (J_E^k\otimes J_H^m)(v_i)\otimes \overline{\alpha^i}
\end{equation*}
for $\xi=\sum_i v_i\otimes \alpha^i$ where $\{v_i\}$ is a frame of $\wedge^kE\otimes S^mH$ and $\alpha^i$ is a $q$-form. 
We denote by $\mathcal{A}^q(\wedge^kE\otimes S^mH)^{\tau}$ the sheaf of $\tau$-invariant elements of $\mathcal{A}^q(\wedge^kE\otimes S^mH)$. 
The map $\tau$ is a real structure of $\mathcal{A}^q(\wedge^kE \otimes S^mH)$ if $k+m$ is even. 
Especially, in the case $k=m=1$, $\tau$ is the complex conjugate on $\mathcal{A}^q(TM\otimes \mathbb{C})$. 

Let $\mathcal{A}^q_{P(H^*)}(\wedge^k E)$ denote the inverse image $p^{-1}p_*\mathcal{A}^q_{P(H^*)}(p^{-1}(\wedge^k E))$ of the direct image of the sheaf $\mathcal{A}^q_{P(H^*)}(p^{-1}(\wedge^k E))$. 
If we take a frame $\{e_i\}$ of $\wedge^kE$, 
then any element $\beta$ of $\mathcal{A}^q_{P(H^*)}(\wedge^k E)$ is written by $\beta=\sum_i e_i\otimes \alpha^i$ for some $q$-forms $\alpha^i$. 
An endomorphism $\widetilde{\tau}$ of $\mathcal{A}^q_{P(H^*)}(\wedge^k E)$ is defined as 
$\widetilde{\tau}(\beta)=\sum_i J_E^k(e_i)\otimes \overline{R_j^*\alpha_i}$. 
We also denote $\widetilde{\tau}(\beta)$ by $J_E^k \overline{R_j^*\beta}$. 
Moreover, we extend the endomorphism $\widetilde{\tau}$ to that of $\mathcal{A}^q_{P(H^*)}(\wedge^k E \otimes S^m\mathbb{H})$ 
by 
\begin{equation*}\label{s3.1.5eq1}
\widetilde{\tau}(\beta \otimes 1^{m-i}j^i)=J_E^k\overline{R_j^*\beta}\otimes 1^{m-i}j^i
\end{equation*}
for $\beta \otimes 1^{m-i}j^i\in \mathcal{A}^q_{P(H^*)}(\wedge^k E \otimes S^m\mathbb{H})$. 
The map $\widetilde{\tau}$ is anti-$\mathbb{C}$-linear and $\widetilde{\tau}^2=(-1)^kR^*_{-1}$. 
\begin{prop}\label{s3.1.5p1} 
The map $\widetilde{\tau}$ defines endomorphisms of $\widetilde{\mathcal{A}}^q_m(\wedge^k E \otimes S^m\mathbb{H})$ and $\widetilde{\mathcal{A}}^q_{(m-i,i)}(\wedge^k E)$ 
such that $\widetilde{\tau}(\widetilde{\xi})=\widetilde{\tau(\xi)}$ and $\widetilde{\tau}(\widetilde{\xi}_i)=\widetilde{\tau(\xi)}_i$ for $\xi \in \mathcal{A}^q(\wedge^k E \otimes S^mH)$. 
Moreover, $\widetilde{\tau}$ induces real structures on $\widetilde{\mathcal{A}}^q_m(\wedge^k E \otimes S^m\mathbb{H})$ and $\widetilde{\mathcal{A}}^q_{(m-i,i)}(\wedge^k E)$ 
if $k+m$ is even. 
\end{prop}
\begin{proof} 
Let $\widetilde{\xi}$ be an element of $\widetilde{\mathcal{A}}^q_m(S^m\mathbb{H})$. 
Then $R^*_{z+jw}\widetilde{\tau}(\widetilde{\xi})=R^*_{z+jw}(\overline{R_j^*\widetilde{\xi}})=\overline{R^*_{z+jw}R_j^*\widetilde{\xi}}
=\overline{R^*_jR_{\bar{z}+j\bar{w}}^*\widetilde{\xi}}=\overline{R^*_j\rho_m(\bar{z}+j\bar{w})\widetilde{\xi}}=\rho_m(z+jw)\overline{R^*_j\widetilde{\xi}}$. 
Hence $\widetilde{\tau}(\widetilde{\xi})$ is in $\widetilde{\mathcal{A}}^q_m( S^m\mathbb{H})$. 
Let $f$ be a polynomial of degree $(m-i,i)$ along fiber.  
Since $(R_j^*f)(u_0^*(z+jw))=f(u_0^*(-\bar{w}+j\bar{z}))$, 
$\widetilde{\tau}(f)$ is a polynomial of $z,w,\bar{z},\bar{w}$ of degree $m$ 
Furthermore, $(R_c)^*\widetilde{\tau}(f)=(R_c)^*\overline{R_j^*f}
=\overline{R_j^*R^*_{\bar{c}}f}=\overline{R_j^*(\bar{c}^{m-i}c^{i}f)}
=c^{m-i}\bar{c}^{i}\overline{R_j^*f}=c^{m-i}\bar{c}^{i}\widetilde{\tau}(f)$ for $c\in \textrm{GL}(1,\mathbb{C})$. 
Hence $\widetilde{\tau}(f)$ is in $\widetilde{\mathcal{A}}^q_{(m-i,i)}$. 
By the bundle map $J^k_E : \wedge^k E\to \wedge^k E$, 
$\widetilde{\tau}$ is extended to endomorphisms of $\widetilde{\mathcal{A}}^q_m(\wedge^k E \otimes S^m\mathbb{H})$ and $\widetilde{\mathcal{A}}^q_{(m-i,i)}(\wedge^k E)$. 
For $\widetilde{\xi}=\sum_i \widetilde{\xi}_i\otimes 1^{m-i}j^i$, 
\begin{equation*}\label{s3.1.5eq5}
\widetilde{\tau}(\widetilde{\xi})=\sum_i J_E^k \overline{R_j^*\widetilde{\xi}_i}\otimes 1^{m-i}j^i
=\sum_i J_E^k \overline{\widetilde{\xi}_i}\otimes \rho_m(j)1^{m-i}j^i=\sum_i J_E^k \overline{\widetilde{\xi}_i}\otimes (-1)^ij^{m-i}1^i. 
\end{equation*}
On the other hand, $\xi$ is given by $\xi=\sum_i \xi_i\otimes u^{m-i}(ju)^i$ for a frame $u$ of $H$. 
Then $\tau (\xi)=\sum_i J_E^k \overline{\xi_i}\otimes J_H^m(u^{m-i}(ju)^i)=\sum_i J_E^k \overline{\xi_i}\otimes (-1)^i(ju)^{m-i}u^i$.
Thus $\widetilde{\tau(\xi)}=\widetilde{\tau}(\widetilde{\xi})$. 
By the definition, $\widetilde{\tau}(\widetilde{\xi})_i=J_E^k \overline{R_j^*\widetilde{\xi}_i}=\widetilde{\tau}(\widetilde{\xi}_i)$. 
Hence $\widetilde{\tau}(\widetilde{\xi}_i)=\widetilde{\tau(\xi)}_i$. 
It follows from $\widetilde{\tau}^2=(-1)^kR^*_{-1}$ that $\widetilde{\tau}^2=(-1)^{k+m}\id$ 
on $\widetilde{\mathcal{A}}^q_m(\wedge^k E \otimes S^m\mathbb{H})$ and $\widetilde{\mathcal{A}}^q_{(m-i,i)}(\wedge^k E)$. 
If $k+m$ is even, the anti-$\mathbb{C}$-linear $\widetilde{\tau}$ is real structures 
on $\widetilde{\mathcal{A}}^q_m(\wedge^k E \otimes S^m\mathbb{H})$ and $\widetilde{\mathcal{A}}^q_{(m-i,i)}(\wedge^k E)$. 
It completes the proof. 
\end{proof} 
We also obtain 
\begin{lem}\label{s3.1.5l1} 
Let $\widetilde{\xi}$ be an element of $\widetilde{\mathcal{A}}^q_m(\wedge^k E \otimes S^m\mathbb{H})$. 
Under the representation (\ref{s3.1eq6}), $\widetilde{\xi}$ is $\widetilde{\tau}$-invariant 
if and only if $\widetilde{\xi}_i$ is $\widetilde{\tau}$-invariant for each $i$, 
and $\widetilde{\xi}_i=(-1)^{m-i}J_E\overline{\widetilde{\xi}_{m-i}}$. $\hfill\Box$
\end{lem}

Let $\widetilde{\mathcal{A}}^q_m(\wedge^k E \otimes S^m\mathbb{H})^{\widetilde{\tau}}$ and 
$\widetilde{\mathcal{A}}^q_{(m,0)}(\wedge^k E)^{\widetilde{\tau}}$ denote 
the sheaves of $\widetilde{\tau}$-invariant elements of $\widetilde{\mathcal{A}}^q_m(\wedge^k E \otimes S^m\mathbb{H})$ and 
$\widetilde{\mathcal{A}}^q_{(m,0)}(\wedge^k E)$, respectively. 
Corollary~\ref{s3.1c2} and Proposition~\ref{s3.1.5p1} imply the following corollary : 
\begin{cor}\label{s3.1.5c1} 
$\mathcal{A}^q(\wedge^kE\otimes S^mH)^{\tau}\cong \widetilde{\mathcal{A}}^q_m(\wedge^k E\otimes S^m\mathbb{H})^{\widetilde{\tau}}
\cong \widetilde{\mathcal{A}}^q_{(m,0)}(\wedge^k E)^{\widetilde{\tau}}$ 
by $\xi\mapsto \widetilde{\xi} \mapsto \widetilde{\xi}_0$. $\hfill\Box$
\end{cor}

\subsection{Canonical 1-form on $P(H^*)$}
We define a $p^{-1}(E)\otimes \mathbb{H}$-valued 1-form $\widetilde{\theta}$ on $P(H^*)$ as 
\[
\widetilde{\theta}_{u^*}(v)=u^{-1}(p_*(v))
\]
for $v\in T_{u^*}P(H^*)$ at $u^*$. 
The 1-form $\widetilde{\theta}$ is called \textit{the canonical 1-form} on $P(H^*)$. 
The identity map $\id_{TM}$ of $TM$ is regarded as an element $\theta$ of $\mathcal{A}^1(E\otimes H)$. 
The canonical form $\widetilde{\theta}\in \widetilde{\mathcal{A}}^1_1(E\otimes \mathbb{H})$ is the lift of $\theta \in \mathcal{A}^1(E\otimes H)$. 
We define $p^{-1}(E)$-valued 1-forms $\widetilde{\theta}_0$ and $\widetilde{\theta}_1$ on $P(H^*)$ as 
\[
\widetilde{\theta}=\widetilde{\theta}_0+\widetilde{\theta}_1 j.
\]
Then $\widetilde{\theta}_0\in \widetilde{\mathcal{A}}^{1}_{(1,0)}(E)$ and $\widetilde{\theta}_1\in \widetilde{\mathcal{A}}^{1}_{(0,1)}(E)$. 
Since $\theta=\id_{TM}$ is $\tau$-invariant, 
the lift $\widetilde{\theta}$ is $\widetilde{\tau}$-invariant. 
Lemma~\ref{s3.1.5l1} implies that $\widetilde{\theta}_0$ and $\widetilde{\theta}_1$ are $\widetilde{\tau}$-invariant, and $\widetilde{\theta}_1=J_E\overline{\widetilde{\theta}_0}$. 

Let $A$ denote the connection form of $P(H^*)$. 
We consider $H^*$ as the right $\mathbb{C}$-bundle, that is, the Lie algebra $gl(1,\mathbb{H})=\mathbb{H}$ is identified with $\mathbb{C}+j\mathbb{C}$.
Then $A$ is written by 
\[
A=\eta_0+j\eta_1
\]
for complex valued 1-forms $\eta_0, \eta_1$ on $P(H^*)$. 
The connection form $A$ is $\widetilde{\tau}$-invariant since $R_j^*A=ad(j)A=\overline{A}$. 
It yields that $\eta_0$ and $\eta_1$ are $\widetilde{\tau}$-invariant. 
\begin{prop}\label{s3.2p1} 
$d^{E}\widetilde{\theta}_0 =-\widetilde{\theta}_0\wedge \eta_0- \eta_1\wedge \widetilde{\theta}_1, \ 
d^{E}\widetilde{\theta}_1 =-\widetilde{\theta}_0\wedge \overline{\eta}_1 - \widetilde{\theta}_1\wedge \overline{\eta}_0$. 
\end{prop}
\begin{proof} 
The element $d_{\widetilde{\mathcal{H}}}\widetilde{\theta}\in \widetilde{\mathcal{A}}_2^2(E\otimes \mathbb{H})$ 
corresponds to the torsion $d^{\nabla}\theta \in \mathcal{A}^{2}(E\otimes H)$. 
It yields that $d_{\widetilde{\mathcal{H}}}\widetilde{\theta}=0$. 
It follows from $d_{\widetilde{\mathcal{H}}}\widetilde{\theta}=d^E\widetilde{\theta} +\widetilde{\theta}\wedge A$ that $d^{E}\widetilde{\theta} =-\widetilde{\theta}\wedge A$. 
We obtain
\[
d^{E}\widetilde{\theta}
=-(\widetilde{\theta}_0+\widetilde{\theta}_1 j)\wedge (\eta_0+j\eta_1) 
=-\widetilde{\theta}_0\wedge \eta_0 +\widetilde{\theta}_1\wedge \eta_1-(\widetilde{\theta}_0\wedge \overline{\eta}_1 +\widetilde{\theta}_1\wedge \overline{\eta}_0)j. 
\]
Hence we finish the proof. 
\end{proof}

Let $s^2_{H}$ denote the symmetrization $\otimes^2 H\to S^2 H$. 
We define an $S^2H$-valued 2-form $\omega$ on $M$ as 
\[
\omega =\omega_{E}\otimes s^2_H. 
\]
Then the lift $\widetilde{\omega}\in \widetilde{\mathcal{A}}^{2}_{2}(S^2\mathbb{H})$ is decomposed by 
\[
\widetilde{\omega}=\widetilde{\omega}_0\, 1\cdot 1+\widetilde{\omega}_1\, 1\cdot j+ \widetilde{\omega}_2\, j\cdot j
\] 
for $\widetilde{\omega}_0\in \widetilde{\mathcal{A}}^2_{(2,0)}$, $\widetilde{\omega}_1\in \widetilde{\mathcal{A}}^2_{(1,1)}$ and $\widetilde{\omega}_2\in \widetilde{\mathcal{A}}^2_{(0,2)}$. 
The $S^2H$-valued 2-form $\omega$ is $\tau$-invariant since $\tau(\omega_{E}\otimes s^2_H)=J_{E^*}^2(\omega_{E})\otimes (J_H^2\otimes J_{H^*}^2)(s^2_{H})=\omega_{E}\otimes s^2_{H}$. 
It follows from Lemma~\ref{s3.1.5l1} that $\widetilde{\omega}_0, \widetilde{\omega}_1$ and $\widetilde{\omega}_2$ are $\widetilde{\tau}$-invariant, 
$\widetilde{\omega}_2=\overline{\widetilde{\omega}_0}$ and $\widetilde{\omega}_1=-\overline{\widetilde{\omega}_1}$.
Since $\widetilde{\omega}_{u^*}(v,w)=p^*\omega_{E}\otimes s^2_{\mathbb{H}}(\widetilde{\theta}_{u^*}(v), \widetilde{\theta}_{u^*}(w))$ for $v,w\in T_{u^*}P(H^*)$ at $u^*$, 
\[
\widetilde{\omega}_0=\omega_{E}(\widetilde{\theta}_0, \widetilde{\theta}_0), \quad 
\widetilde{\omega}_1=\omega_{E}(\widetilde{\theta}_0, \widetilde{\theta}_1) +\omega_{E}(\widetilde{\theta}_1, \widetilde{\theta}_0), \quad 
\widetilde{\omega}_2=\omega_{E}(\widetilde{\theta}_1, \widetilde{\theta}_1). 
\]

\subsection{Curvature on $P(H^*)$}
Let $\{I, J, K\}$ be a local basis of $Q$ for a $\mathbb{C}$-frame $\{u, ju\}$ of $H$ as in (\ref{s2eq3}). 
The dual bundle $Q^*$ of $Q$ is contained in the real form of $\textrm{End}_{\mathbb{C}}(H^*)=H^*\otimes_{\mathbb{C}}H$ with respect to $J_{H^*}\otimes J_H$. 
The Killing form $B_0$ on $sp(1)$ induces a symmetric bi-linear form $B$ on $Q^*$. 
\begin{lem}\label{s3.3l1} 
$B=-2(I\otimes I+J\otimes J+K\otimes K)$. 
\end{lem}
\begin{proof} 
Let $I^*, J^*, K^*$ be the dual of endomorphisms $I, J, K$, respectively. 
The dual basis $\{(I^*)^{\vee}, (J^*)^{\vee}, (K^*)^{\vee}\}$ of $\{I^*, J^*, K^*\}$ satisfies $(I^*)^{\vee}=-\frac{1}{2}I, (J^*)^{\vee}=-\frac{1}{2}J$ and $(K^*)^{\vee}=-\frac{1}{2}K$. 
The Killing form $B_0$ is given by $B_0=-8(i^{\vee}\otimes i^{\vee}+j^{\vee}\otimes j^{\vee}+k^{\vee}\otimes k^{\vee})$ on $sp(1)=\left< i,j,k \right>$. 
Hence $B=-8((I^*)^{\vee}\otimes (I^*)^{\vee}+(J^*)^{\vee}\otimes (J^*)^{\vee}+(K^*)^{\vee}\otimes (K^*)^{\vee})=-2(I\otimes I+J\otimes J+K\otimes K)$. 
\end{proof}
The endomorphisms $I, J, K$ of $H$ induce almost complex structures on $M$, locally. 
We define local 2-forms $\omega_I$, $\omega_J$ and $\omega_K$ on $M$ 
by $\omega_I(X,Y)=g(IX,Y)$, $\omega_J(X,Y)=g(JX,Y)$ and $\omega_K(X,Y)=g(KX,Y)$ for $X,Y\in TM$. 
They are written by 
\begin{align*}
\omega_I&=\omega_E\otimes \frac{i}{|u^*|^2}(u^*\otimes (ju)^*+(ju)^*\otimes u^*), \quad \omega_J=\omega_E\otimes \frac{-1}{|u^*|^2}(u^*\otimes u^*+(ju)^*\otimes (ju)^*), \\
\omega_K&=\omega_E\otimes \frac{-i}{|u^*|^2}(u^*\otimes u^*-(ju)^*\otimes (ju)^*) 
\end{align*}
for the $\mathbb{C}$-frame $\{u, ju\}$ of $H$. 
We define $B^{\sharp}$ by a global $Q$-valued $2$-form 
$B^{\sharp}=-2(I\otimes \omega_I+J\otimes \omega_J+K\otimes \omega_K)$ on $M$. 

Let $\Omega$ be the curvature form of $P(H^*)$. 
We define a function $r$ on $P(H^*)$ by 
\[
r(u^*)=|u^*|
\]
for $u^*\in P(H^*)$, where $|\cdot|$ means the norm of $H^*$. 
\begin{prop}\label{s3.3p1} 
Let $t$ be the scalar curvature of $M$. 
Then $\Omega=-2c_ntr^{-2}(\widetilde{\omega}_1+2j\widetilde{\omega}_0)$ where $c_n$ is a positive number depending only on $n$. 
\end{prop}
\begin{proof} 
The curvature $R_H$ of $H$ is given by $c_n tB^{\sharp}$ for a positive number $c_n$ depending on $n$ (c.f. \cite{A}, \cite{S1}). 
Then $R_{H^*}=-c_nt(B^{\sharp})^*$ where $(B^{\sharp})^*=-2(I^*\otimes \omega_I+J^*\otimes \omega_J+K^*\otimes \omega_K)$. 
Hence $R_{H^*}=2c_nt(I^*\otimes \omega_I+J^*\otimes \omega_J+K^*\otimes \omega_K)$. 
The curvature form $\Omega$ of $H^*$ is given by the $sp(1)$-valued 2-form 
\begin{equation*}\label{s3.3eq5}
\Omega=2c_nt(i\otimes \omega_I+j\otimes \omega_J+k\otimes \omega_K).
\end{equation*}
We take $v, w\in T_{u^*}P(H^*)$. 
Then $p_*(v)=e_0\otimes u+e_1\otimes ju$ and $p_*(w)=e'_0\otimes u+e'_1\otimes ju$ for $e_i,  e'_i \in E_{p(u^*)}$. 
An easy calculation shows that 
\begin{align*}
p^*\omega_I(v,w)&=i|u^*|^{-2}(\omega_E(e_0,e'_1)+\omega_E(e_1,e'_0))=i|u^*|^{-2}(\widetilde{\omega}_1)_{u^*}(v,w), \\
p^*\omega_J(v,w)-i p^*\omega_K(v,w)&=2|u^*|^{-2}\omega_E(e_0,e'_0)=-2|u^*|^{-2}(\widetilde{\omega}_0)_{u^*}(v,w) 
\end{align*}
at $u^*$. 
Thus we obtain 
\begin{equation}\label{s3.3eq8}
i\omega_I=-r^{-2}\widetilde{\omega}_1, \quad \omega_J-i\omega_K=-2r^{-2}\widetilde{\omega}_0 
\end{equation}
on $P(H^*)$. Hence $\Omega=2c_nt(i\omega_I+j\omega_J+k\omega_K)=-2c_ntr^{-2}(\widetilde{\omega}_1+2j\widetilde{\omega}_0)$. 
\end{proof}

From now on, we set $c=2c_nt$. 
Then $\Omega=-cr^{-2}(\widetilde{\omega}_1+2j\widetilde{\omega}_0)$. 
By $dA=\Omega -A \wedge A$, we obtain 
\begin{prop}\label{s3.3p2} 
$d\eta_0=-cr^{-2}\widetilde{\omega}_1-\eta_1\wedge \overline{\eta}_1, \ 
d\eta_1=-2cr^{-2}\widetilde{\omega}_0+\eta_0 \wedge \eta_1+\eta_1\wedge \overline{\eta}_0$. $\hfill\Box$
\end{prop}

\subsection{Complex structure on $P(H^*)$}
We identify each fiber of $p:P(H^*)\to M$ with $\mathbb{C}^2\backslash \{0\}$ by $\mathbb{H}=\mathbb{C}+j\mathbb{C}\cong \mathbb{C}^2$. 
Then the fiber admits a complex structure. 
Let $\widetilde{\mathcal{V}}$ be a vector bundle consisting of tangent vectors to fibers. 
We denote by $i_{\widetilde{\mathcal{V}}}$ the complex structure of $\widetilde{\mathcal{V}}$. 
The horizontal $\widetilde{\mathcal{H}}\otimes \mathbb{C}$ is isomorphic to $TM\otimes \mathbb{C}=E\otimes H$ by $dp$. 
The element $I$ as in (\ref{s2eq3}) induces a complex structure of $\widetilde{\mathcal{H}}$ 
such that $\widetilde{\mathcal{H}}^{1,0}\cong E\otimes u$ and $\widetilde{\mathcal{H}}^{0,1}\cong E\otimes ju$ at $u^*$, 
denoted also by $I$. 
We define an almost complex structure $\widetilde{I}$ on $P(H^*)$ by 
\[
\widetilde{I}=I+i_{\widetilde{\mathcal{V}}}
\]
under the decomposition $TP(H^*)=\widetilde{\mathcal{H}}\oplus \widetilde{\mathcal{V}}$. 
Then $\eta_0$ and $\eta_1$ are $(1,0)$-forms. 
Actually, we fix a point $x$ of $M$ and take an unitary frame $u_0^*$ of $P(H^*)$ with $\nabla u_0^*=0$ at $x=p(u_0^*)$. 
At $u^*=u_0^*(z+jw)$, 
\begin{equation*}\label{s3.5eq2}
\eta_0=r^{-2}(\bar{z}dz+\bar{w}dw),\quad \eta_1=r^{-2}(-wdz+zdw).
\end{equation*}
By the definition of $\widetilde{\theta}$, the $p^{-1}(E)$-valued 1-forms $\widetilde{\theta}_0$ and $\widetilde{\theta}_1$ are $(1,0)$ and $(0,1)$-forms, respectively. 
\begin{prop}\label{s3.4p1}
(c.f. Theorem~4.1 in \cite{AHS}, Theorem~4.1 in \cite{S1})
The almost complex structure $\widetilde{I}$ is integrable. 
\end{prop}
\begin{proof} 
Taking an open set $U$ of $P(H^*)$ and a local basis $\{e_i\}$ of $E$, 
then we define local 1-forms $\alpha^i, \beta^i$ by 
$\widetilde{\theta}_0=\sum_{i=1}^{2n}e_i\otimes \alpha^i$ and $\widetilde{\theta}_1=\sum_{i=1}^{2n}e_i\otimes \beta^i$. 
Let $D$ denote the distribution $\langle \alpha^1,\alpha^2,\cdots, \alpha^{2n}, \eta_0, \eta_1 \rangle$ on $U$. 
Then $\wedge^{1,0}T^*P(H^*)=D$ on $U$. 
It follows from Proposition~\ref{s3.3p2}, the $(2,0)$-form $\widetilde{\omega}_0$ and the $(1,1)$-form $\widetilde{\omega}_1$ that 
$d\eta_0, d\eta_1 \in D\wedge \mathcal{A}^1$ where $\mathcal{A}^1$ is the sheaf of differential 1-forms on $P(H^*)$. 
Proposition \ref{s3.2p1} implies $d \alpha^i\in D\wedge\mathcal{A}^1$. 
It turns out that $dD\subset D\wedge\mathcal{A}^1$, and hence the almost complex structure is integrable. 
\end{proof}

\begin{rem}\label{s3.4r1}
{\rm 
We take a torsion free connection $\nabla$ of $TP(H^*)$ preserving $\widetilde{I}$. 
Let $F$ be a holomorphic vector bundle on $P(H^*)$ and $\nabla_F$ a $(1,0)$-connection $\nabla_F: F\to F\otimes T^*$ of $F$. 
We consider the connection $\nabla_{F\otimes \wedge^q}$ of $F\otimes \wedge^q$ as the map $F\otimes \wedge^q \to F\otimes \wedge^q\otimes T^*$. 
Then the covariant exterior derivative $d^{\nabla_F}$ is given by $(-1)^q\wedge\circ \nabla_{F\otimes \wedge^q}$. 
The operator $\bar{\partial}_{F}:F\otimes \wedge^{q,0} \to F\otimes \wedge^{q,1}$ satisfies $\bar{\partial}_{F}=(-1)^q\wedge\circ \nabla_{F\otimes \wedge^q}^{0,1}$. 
If $\bar{\partial}_{F}\alpha =\sum_{i,j}\beta_i\wedge \gamma_j$ for $\beta_j\in \wedge^{q,0}, \gamma_i\in \wedge^{0,1}$, 
then $\nabla_{F\otimes \wedge^q}^{0,1} \alpha =(-1)^q\sum_{i,j} \beta_i\otimes \gamma_j$. 
}
\end{rem}

\begin{lem} \label{s3.4l1}
\begin{align*}
\nabla^{0,1} \eta_0&=cr^{-2}\omega_E(\widetilde{\theta}_0,\widetilde{\theta}_1)+\eta_1\otimes \overline{\eta}_1, \\
\nabla^{0,1} \eta_1&=-\eta_1\otimes \overline{\eta}_0, \\
\nabla^{0,1}_{E\otimes \wedge^1} \widetilde{\theta}_0&=\eta_1\otimes \widetilde{\theta}_1. 
\end{align*}
\end{lem}
\begin{proof} 
Proposition \ref{s3.2p1}, \ref{s3.3p2} and \ref{s3.4p1} imply that 
\[
\bar{\partial} \eta_0=-cr^{-2}\widetilde{\omega}_1-\eta_1\wedge \overline{\eta}_1, \quad
\bar{\partial} \eta_1=\eta_1\wedge \overline{\eta}_0, \quad
\bar{\partial}_E \widetilde{\theta}_0=-\eta_1\wedge \widetilde{\theta}_1. 
\]
By Remark~\ref{s3.4r1}, we obtain the formula in this Lemma. 
\end{proof} 

From now on, we write $\nabla^{0,1}\widetilde{\theta}_0$ instead of $\nabla^{0,1}_{E\otimes \wedge^1}\widetilde{\theta}_0$ for simplicity. 
We define a $p^{-1}(\wedge^k E)$-valued $(k,0)$-form $\widetilde{\theta}_0^k$ 
by the $k$-th wedge $\sum_{i_1,\cdots,i_k=1}^{2n}e_{i_1}\wedge \cdots \wedge e_{i_k}\otimes \alpha_{i_1} \wedge\cdots \wedge \alpha_{i_k}$ 
of $\widetilde{\theta}_0=\sum_{i=1}^{2n}e_i\otimes \alpha_i$. 
Lemma \ref{s3.4l1} implies the following : 
\begin{prop} \label{s3.4p2} 
\begin{align*}
\nabla^{0,1} \widetilde{\theta}_0^k&=k\widetilde{\theta}_0^{k-1}\wedge \eta_1 \wedge_E \widetilde{\theta}_1 & \\
\nabla^{0,1} (\widetilde{\theta}_0^{k-1}\wedge \eta_0)&=-(k-1)\widetilde{\theta}_0^{k-2}\wedge \eta_0\wedge \eta_1 \wedge_E\widetilde{\theta}_1 
+\widetilde{\theta}_0^{k-1}\wedge (cr^{-2} \omega_E(\widetilde{\theta}_0,\widetilde{\theta}_1)+\eta_1\otimes \overline{\eta}_1), & \\
\nabla^{0,1} (\widetilde{\theta}_0^{k-1}\wedge \eta_1)&=-\widetilde{\theta}_0^{k-1}\wedge \eta_1\otimes \overline{\eta}_0, & \\
\nabla^{0,1} (\widetilde{\theta}_0^{k-2}\wedge \eta_0\wedge \eta_1)
&=-\widetilde{\theta}_0^{k-2}\wedge \eta_1\wedge(cr^{-2} \omega_E(\widetilde{\theta}_0,\widetilde{\theta}_1)- \eta_0 \otimes \overline{\eta}_0). &\hfill\Box
\end{align*}
\end{prop}

\subsection{Holomorphic symplectic structure and hyperk\"ahler structure on $P(H^*)$}
We recall that the function $r$ on $P(H^*)$ is given by $r(u^*)=|u^*|_{p(u^*)}$ at $u^*\in P(H^*)$. 
\begin{lem}\label{s3.5l1} 
$dr^2=r^2(\eta_0+\overline{\eta}_0),\ \bar{\partial}r^2=r^2\overline{\eta}_0$
\end{lem}
\begin{proof} 
The symplectic form $\omega_{H^*}$ of $H^*$ is the element of $\mathcal{A}^0(\wedge^2 H)$. 
Then $\widetilde{(\omega_{H^*})}=r^2 1\wedge j$ on $P(H^*)$, 
and $d_{\widetilde{\mathcal{H}}}\widetilde{(\omega_{H^*})}=0$ since $\nabla \omega_{H^*}=0$. 
It follows from $d_{\widetilde{\mathcal{H}}}\widetilde{(\omega_{H^*})}=d\widetilde{(\omega_{H^*})}-(\eta_0+\overline{\eta}_0)\widetilde{(\omega_{H^*})}$ 
that $dr^2=r^2(\eta_0+\overline{\eta}_0)$. 
Immediately, it implies that $\bar{\partial}r^2=r^2\overline{\eta}_0$. 
\end{proof} 

The complex manifold $P(H^*)$ has a holomorphic symplectic structure 
if the scalar curvature $t$ of $M$ is not zero as follows : 
\begin{prop}\label{s3.5p1} 
The form $r^2\eta_1$ is a holomorphic $(1,0)$-form on $P(H^*)$. 
If $t\neq 0$, 
then $d(r^2\eta_1)$ is a holomorphic symplectic form on $P(H^*)$. 
\end{prop}
\begin{proof} 
Proposition~\ref{s3.3p2} and  Lemma~\ref{s3.5l1} imply that $d(r^2\eta_1)=2(-c\widetilde{\omega}_0+r^2\eta_0\wedge \eta_1)$. 
It turns out that $d(r^2\eta_1)$ is a (2,0)-form, and so $r^2\eta_1$ is a holomorphic $(1,0)$-form on $P(H^*)$. 
The 2-form $\eta_0\wedge \eta_1$ is non-degenerate on $\widetilde{\mathcal{V}}^{1,0}$. 
If $t\neq 0$, then $c\widetilde{\omega}_0$ is non-degenerate on $\widetilde{\mathcal{H}}^{1,0}$. 
Hence, the 2-form $d(r^2\eta_1)$ is a non-degenerate holomorphic (2,0)-form on $P(H^*)$. 
\end{proof} 

The quaternionic vector space $\mathbb{H}$ is regarded as $\mathbb{C}_j+k\mathbb{C}_j$ where $\mathbb{C}_j$ means a complex vector space with respect to $j$. 
Then a complex structure $j_{\widetilde{\mathcal{V}}}$ of $\widetilde{\mathcal{V}}$ is induced 
by the decomposition of $\mathbb{H}$ as the same manner as $i_{\widetilde{\mathcal{V}}}$. 
Considering the decomposition $\mathbb{H}=\mathbb{C}_k+i\mathbb{C}_k$ where $\mathbb{C}_k$ is a complex vector space with respect to $k$, 
then we also obtain a complex structure $k_{\widetilde{\mathcal{V}}}$ of $\widetilde{\mathcal{V}}$. 
The endomorphisms  $J, K$ of $H$ defined by (\ref{s2eq3}) also induce complex structures of $\widetilde{\mathcal{H}}$, which denote by $J, K$, respectively. 
Then $\widetilde{J}=J+j_{\widetilde{\mathcal{V}}}$, $\widetilde{K}=K+k_{\widetilde{\mathcal{V}}}$ are complex structures on $P(H^*)$ 
by the same argument of Proposition~\ref{s3.4p1}. 
We remark that $\widetilde{I}\widetilde{J}=-\widetilde{K}$. 
Hence $(\widetilde{I}, \widetilde{J}, -\widetilde{K})$ is a hypercomplex structure on $P(H^*)$. 
We define symmetric 2-forms $g_{\widetilde{\mathcal{V}}}$ and $\widetilde{g}$ by 
\[
g_{\widetilde{\mathcal{V}}}=\eta_0\otimes \overline{\eta}_0+\overline{\eta}_0\otimes \eta_0+\eta_1\otimes \overline{\eta}_1+\overline{\eta}_1\otimes \eta_1
\]
and
\[
\widetilde{g}=r^2(cp^*g+g_{\widetilde{\mathcal{V}}})
\] 
on $P(H^*)$. 
\begin{prop}\label{s3.5p1} 
If $t>0$, 
then $(\widetilde{g}, \widetilde{I}, \widetilde{J}, -\widetilde{K})$ is a hyperk\"ahler structure on $P(H^*)$ 
such that $-id(r^2\eta_0), d(r^2\eta_1^{\rm Re}), d(r^2\eta_1^{\rm Im})$ are K\"{a}hler forms with respect to $\widetilde{I}, \widetilde{J}, -\widetilde{K}$, respectively. 
\end{prop}
\begin{proof} 
It follows from Proposition~\ref{s3.3p2},  Lemma~\ref{s3.5l1} and the equation (\ref{s3.3eq8}) that 
\begin{align*} 
-id(r^2\eta_0)&=r^2(c\omega_I+i(\eta_0\wedge \overline{\eta}_0+\eta_1\wedge \overline{\eta}_1)), \\
d(r^2\eta_1)&=r^2(c(\omega_J-i\omega_K)+2\eta_0\wedge \eta_1). 
\end{align*}
It is easy to see that $i(\eta_0\wedge \overline{\eta}_0+\eta_1\wedge \overline{\eta}_1)=g_{\widetilde{\mathcal{V}}}(i_{\widetilde{\mathcal{V}}}\cdot, \cdot)$. 
Hence $-id(r^2\eta_0)$ is a K\"{a}hler form with respect to $(\widetilde{g}, \widetilde{I})$. 
We fix $u_0^*\in P(H^*)$ and provide a coordinate $(z,w)$ of fiber such that $u^*=u_0^*(z+jw)$. 
Then $j$ maps $(z,w)$ to $(-\bar{w}, \bar{z})$. 
Since $dz\circ j_{\widetilde{\mathcal{V}}}=-d\bar{w}$ and $dw\circ j_{\widetilde{\mathcal{V}}}=d\bar{z}$, 
we obtain $\eta_0\circ j_{\widetilde{\mathcal{V}}}=-\overline{\eta}_1$ and $\eta_1\circ j_{\widetilde{\mathcal{V}}}=\overline{\eta}_0$. 
By the same manner, $\eta_0\circ k_{\widetilde{\mathcal{V}}}=i\overline{\eta}_1$ and $\eta_1\circ k_{\widetilde{\mathcal{V}}}=-i\overline{\eta}_0$. 
It yields that $g_{\widetilde{\mathcal{V}}}(j_{\widetilde{\mathcal{V}}}\cdot, \cdot)=\eta_0\wedge \eta_1+\overline{\eta}_0\wedge \overline{\eta}_1$ 
and $g_{\widetilde{\mathcal{V}}}(k_{\widetilde{\mathcal{V}}}\cdot, \cdot)=i(\eta_0\wedge \eta_1-\overline{\eta}_0\wedge \overline{\eta}_1)$. 
Thus $g_{\widetilde{\mathcal{V}}}(j_{\widetilde{\mathcal{V}}}\cdot, \cdot)-ig_{\widetilde{\mathcal{V}}}(k_{\widetilde{\mathcal{V}}}\cdot, \cdot)=2\eta_0\wedge \eta_1$. 
It turns out that $d(r^2\eta_1)^{\rm Re}=d(r^2\eta_1^{\rm Re})$ and $d(r^2\eta_1)^{\rm Im}=d(r^2\eta_1^{\rm Im})$ 
are K\"{a}hler forms with respect to $\widetilde{J}$ and $-\widetilde{K}$, respectively. 
Hence we finish the proof. 
\end{proof} 
The hyperk\"ahler structure $(\widetilde{g}, \widetilde{I}, \widetilde{J}, -\widetilde{K})$ induces that on $P(H^*)/\mathbb{Z}_2$. 
This coincides the hyperk\"ahler structure constructed by Swann~\cite{Sw}.

\subsection{Holomorphic $k$-vector fields on $P(H^*)$}
Let $\widetilde{\mathcal{H}}^{1,0}$ and $\widetilde{\mathcal{V}}^{1,0}$ be the vector bundles of horizontal $(1,0)$-vectors and vertical $(1,0)$-vectors on $P(H^*)$. 
Then $\widetilde{\mathcal{H}}^{1,0}=\Ker \eta_0 \cap \Ker \eta_1$. 
There exist fundamental vector fields $\hat{1}, \hat{i},\hat{j}, \hat{k}$ associated with the elements $1,i,j,k$ of Lie algebra $gl(1,\mathbb{H})=\mathbb{H}$. 
We define the complex vector fields $v_0$ and $v_1$ as $v_0=\frac{1}{2}(\hat{1}-i \hat{i})$ and $v_1=\frac{1}{2}(\hat{j}-i \hat{k})$. 
Then $\{ v_0, v_1\}$ is the dual basis of $\{ \eta_0, \eta_1\}$.  
The two vector fields $v_0$ and $v_1$ span the space $\widetilde{\mathcal{V}}^{1,0}$. 
Let $X'$ be a $(1,0)$-vector field on $P(H^*)$. 
Then $X'$ is decomposed into 
\begin{equation}\label{s5.2eq1} 
X'=X'_h+f_0v_0+f_1v_1
\end{equation}
for a horizontal vector field $X'_h$ and functions $f_0, f_1$ on $P(H^*)$. 
\begin{lem}\label{s5.2l1} 
The $(1,0)$-vector field $X'$ is holomorphic if and only if 
\begin{align*}
\mathrm{(i)}&\quad \bar{\partial}(\widetilde{\theta}_0(X'_h))-f_1\widetilde{\theta}_1=0, \\
\mathrm{(ii)}&\quad \bar{\partial}f_0= cr^{-2}\omega_E(\widetilde{\theta}_0(X'_h), \widetilde{\theta}_1)+ f_1\overline{\eta}_1 
\end{align*}
under the decomposition (\ref{s5.2eq1}). 
\end{lem}
\begin{proof} 
The vector field $X'$ is holomorphic if and only if $\nabla^{0,1}X'=0$. 
The equation is equal to $\widetilde{\theta}_0(\nabla^{0,1}X')=0$, 
$\eta_0(\nabla^{0,1}X')=0$ and 
$\eta_1(\nabla^{0,1}X')=0$. 
The first equation induces the third one since 
$\bar{\partial}^{\nabla}(\widetilde{\theta}_0(\nabla^{0,1}X'))
=\eta_1(\nabla^{0,1}X')\wedge \widetilde{\theta}_1+\widetilde{\theta}_0(\Omega_{TP(H^*)}^{(0,2)}(X'))
=\eta_1(\nabla^{0,1}X')\wedge \widetilde{\theta}_1$ 
and the map $\wedge \widetilde{\theta}_1 : \wedge^{0,1} \to p^{-1}(E)\otimes  \wedge^{0,2}$ is injective. 
Lemma~\ref{s3.4l1} implies that 
$\widetilde{\theta}_0(\nabla^{0,1}X') =\bar{\partial}(\widetilde{\theta}_0(X'_h)) - f_1 \widetilde{\theta}_1$ and 
$\eta_0(\nabla^{0,1}X') =\bar{\partial}f_0-cr^{-2}\omega_E(\widetilde{\theta}_0(X'_h), \widetilde{\theta}_1)- f_1\overline{\eta}_1$. 
It turns out that $\nabla^{0,1}X'=0$ is equivalent to the conditions (i) and (ii). 
\end{proof} 

Let $k$ be an integer which is greater than $1$. 
Any $(k,0)$-vector $X'$ is decomposed into 
\begin{equation}\label{s5.3eq1}
X'=X'_h+Y_0\wedge v_0+Y_1\wedge v_1+ Z_0\wedge v_0\wedge v_1
\end{equation}
for $X'_h\in \wedge^k \widetilde{\mathcal{H}}^{1,0}$ and $Y_0, Y_1 \in\wedge^{k-1} \widetilde{\mathcal{H}}^{1,0}$ and $Z_0\in \wedge^{k-2} \widetilde{\mathcal{H}}^{1,0}$. 
\begin{lem}\label{s5.3l2} 
For $2\le k\le 2n$, the $(k,0)$-vector field $X'$ is holomorphic if and only if 
\begin{align*}
\mathrm{(i)}&\quad \bar{\partial}(\widetilde{\theta}_0^k(X'_h))-k^2\widetilde{\theta}_0^{k-1}(Y_1)\wedge_E \widetilde{\theta}_1=0, \\
\mathrm{(ii)}&\quad k^2\bar{\partial}(\widetilde{\theta}_0^{k-1}(Y_0))+k^2(k-1)^2\widetilde{\theta}_0^{k-2}(Z_0)\wedge_E \widetilde{\theta}_1 
- cr^{-2}\omega_E(\widetilde{\theta}_0^k(X'_h), \widetilde{\theta}_1)-k^2\widetilde{\theta}_0^{k-1}(Y_1)\otimes \overline{\eta}_1=0, \\
\mathrm{(iii)}&\quad \bar{\partial}(\widetilde{\theta}_0^{k-1}(Y_1))+\widetilde{\theta}_0^{k-1}(Y_1)\otimes \overline{\eta}_0=0, \\
\mathrm{(iv)}&\quad (k-1)^2\bar{\partial}(\widetilde{\theta}_0^{k-2}(Z_0))+(k-1)^2\widetilde{\theta}_0^{k-2}(Z_0)\otimes \overline{\eta}_0- cr^{-2}\omega_E(\widetilde{\theta}_0^{k-1}(Y_1), \widetilde{\theta}_1)=0 
\end{align*}
under the decomposition (\ref{s5.3eq1}). 
Especially, in the case $k\neq 2n$, $X'$ is holomorphic 
if and only if the equations $(i), (ii), (iv)$ hold. 
\end{lem}
\begin{proof} 
The equation $\nabla^{0,1}X'=0$ is equal to $\widetilde{\theta}_0^k(\nabla^{0,1}X')=0$, 
$(\widetilde{\theta}_0^{k-1}\wedge \eta_0)(\nabla^{0,1}X')=0$, 
$(\widetilde{\theta}_0^{k-1}\wedge \eta_1)(\nabla^{0,1}X')=0$ and 
$(\widetilde{\theta}_0^{k-2}\wedge \eta_0 \wedge \eta_1)(\nabla^{0,1}X')=0$. 
Proposition~\ref{s3.4p2} implies that 
$\widetilde{\theta}_0^k(\nabla^{0,1}X') =\bar{\partial}(\widetilde{\theta}_0^k(X'_h)) - k^2\widetilde{\theta}_0^{k-1}(Y_1)\wedge_E \widetilde{\theta}_1$ 
and $\widetilde{\theta}_0^{k-1}\wedge \eta_1(\nabla^{0,1}X') =k(\bar{\partial}(\widetilde{\theta}_0^{k-1}(Y_1))+\widetilde{\theta}_0^{k-1}(Y_1)\otimes \overline{\eta}_0)$. 
By $\omega_E(\widetilde{\theta}_0^{k}(X'_h), \widetilde{\theta}_1)=k(\widetilde{\theta}_0^{k-1}\wedge \omega_E(\widetilde{\theta}_0,\widetilde{\theta}_1))(X')$, 
we obtain 
\begin{align*} 
&\widetilde{\theta}_0^{k-1}\wedge \eta_0(\nabla^{0,1}X') \\
&=k\bar{\partial}(\widetilde{\theta}_0^{k-1}(Y_0)) +((k-1)\widetilde{\theta}_0^{k-2}\wedge \eta_0\wedge \eta_1 \wedge_E\widetilde{\theta}_1 
-\widetilde{\theta}_0^{k-1}\wedge (cr^{-2} \omega_E(\widetilde{\theta}_0,\widetilde{\theta}_1)+\eta_1\otimes \overline{\eta}_1))(X')\\
&=k\bar{\partial}(\widetilde{\theta}_0^{k-1}(Y_0))+k(k-1)^2\widetilde{\theta}_0^{k-2}(Z_0)\wedge_E \widetilde{\theta}_1 
-\frac{1}{k} cr^{-2}\omega_E(\widetilde{\theta}_0^{k}(X'_h), \widetilde{\theta}_1)-k\widetilde{\theta}_0^{k-1}(Y_1)\otimes \overline{\eta}_1 
\end{align*}
and 
\begin{align*} 
&\widetilde{\theta}_0^{k-2}\wedge \eta_0\wedge \eta_1(\nabla^{0,1}X') \\
&=k(k-1)\bar{\partial}(\widetilde{\theta}_0^{k-2}(Z_0))- kcr^{-2}\widetilde{\theta}_0^{k-2}\wedge \omega_E(\widetilde{\theta}_0, \widetilde{\theta}_1)(Y_1)+k(k-1)\widetilde{\theta}_0^{k-2}(Z_0)\otimes \overline{\eta}_0 \\
&=k(k-1)\bar{\partial}(\widetilde{\theta}_0^{k-2}(Z_0))- \frac{k}{k-1}cr^{-2}\omega_E(\widetilde{\theta}_0^{k-1}(Y_1), \widetilde{\theta}_1)+k(k-1)\widetilde{\theta}_0^{k-2}(Z_0)\otimes \overline{\eta}_0. 
\end{align*}
Therefore, $\nabla^{0,1}X'=0$ if and only if $(i), (ii), (iii), (iv)$ hold. 
Now $\overline{\partial}^{\nabla}(\widetilde{\theta}_0^k(\nabla^{0,1}X'))=(-1)^{k+1}(\widetilde{\theta}_0^{k-1}\wedge \eta_1)(\nabla^{0,1}X')\wedge \widetilde{\theta}_1$. 
The map $\wedge \widetilde{\theta}_1 : p^{-1}(\wedge^{k-1} E)\otimes \wedge^{0,1} \to p^{-1}(\wedge^k E)\otimes \wedge^{0,2}$ is injective in the case $1\le k \le 2n-1$. 
Hence $\widetilde{\theta}_0^k(\nabla^{0,1}X')=0$ implies that 
$(\widetilde{\theta}_0^{k-1}\wedge \eta_1)(\nabla^{0,1}X')=0$ for $1\le k \le 2n-1$. 
It turns out that $\nabla^{0,1}X'=0$ is equivalent to equations $(i)$, $(ii)$ and $(iv)$ in the case $1\le k \le 2n-1$. 
\end{proof} 

From now on, we extend the decomposition (\ref{s5.3eq1}) to the case $k=1$ as $Z_0=0$. 
\begin{prop}\label{s5.3p2} 
Let $k$ be an integer with $1\le k\le 2n-1$. 
If $X'_h\in \wedge^k \widetilde{\mathcal{H}}^{1,0}$ and $Y_1\in \wedge^{k-1} \widetilde{\mathcal{H}}^{1,0}$ satisfy 
\[
\mathrm{(i)}\quad \bar{\partial}(\widetilde{\theta}_0^k(X'_h))-k^2\widetilde{\theta}_0^{k-1}(Y_1)\wedge_E \widetilde{\theta}_1=0, 
\]
then there exist $Y_0$ and $Z_0$ locally such that $X'$ is holomorphic. 
If $X'_h\in \wedge^{2n} \widetilde{\mathcal{H}}^{1,0}$ and $Y_1\in \wedge^{2n-1} \widetilde{\mathcal{H}}^{1,0}$ satisfy $(i)$ and 
\[
\mathrm{(iii)}\quad \bar{\partial}(\widetilde{\theta}_0^{2n-1}(Y_1))+\widetilde{\theta}_0^{2n-1}(Y_1)\otimes \overline{\eta}_0=0, 
\]
then there exist $Y_0$ and $Z_0$ locally such that $X'$ is holomorphic. 
\end{prop}
\begin{proof} 
We assume that $X'_h\in \wedge^k \widetilde{\mathcal{H}}^{1,0}$ and $Y_1\in \wedge^{k-1} \widetilde{\mathcal{H}}^{1,0}$ satisfy 
(i) for $1\le k\le 2n-1$, (i) and (iii) for $k=2n$. 
It follows from Proposition~\ref{s3.2p1} and Lemma~\ref{s3.5l1} that $\bar{\partial}(r^{-2}\widetilde{\theta}_1)=0$. 
By taking the derivative $\bar{\partial}$ on (i), 
we obtain $\bar{\partial}(r^2\widetilde{\theta}_0^{k-1}(Y_1))\wedge \widetilde{\theta}_1=0$. 
Since the wedge $\wedge \widetilde{\theta}_1 : p^{-1}(\wedge^{k-1} E)\to p^{-1}(\wedge^k E) \otimes T^*P(H^*)$ is injective for $1\le k\le 2n-1$, 
$\bar{\partial}(r^2\widetilde{\theta}_0^{2n-1}(Y_1))=0$. 
The equation (iii) is equal to $\bar{\partial}(r^2\widetilde{\theta}_0^{2n-1}(Y_1))=0$. 
Hence $\bar{\partial}(r^2\widetilde{\theta}_0^{k-1}(Y_1))=0$ for $1\le k\le 2n$. 
It is easy to see that the condition (iv) in Lemma~\ref{s5.3l2} is equivalent to 
\begin{equation}\label{s5.3eq4}
\bar{\partial}((k-1)^2r^2\widetilde{\theta}_0^{k-2}(Z_0))= c\omega_E(r^2\widetilde{\theta}_0^{k-1}(Y_1), r^{-2}\widetilde{\theta}_1).
\end{equation}
The derivative $\bar{\partial}$ on the right hand side of (\ref{s5.3eq4}) vanishes 
since $\bar{\partial}(r^2\widetilde{\theta}_0^{k-1}(Y_1))=0$ and $\bar{\partial}(r^{-2}\widetilde{\theta}_1)=0$. 
By Dolbeaux's lemma, there exists an element $\zeta \in \mathcal{A}^0_{P(H^*)}(\wedge^{k-2}E)$ 
such that $\bar{\partial}\zeta= c\omega_E(r^2\widetilde{\theta}_0^{k-1}(Y_1), r^{-2}\widetilde{\theta}_1)$. 
The $(k-2)$-th wedge $\widetilde{\theta}_0^{k-2}$ is an isomorphism 
from $\mathcal{A}^0_{P(H^*)}(\wedge^{k-2}\widetilde{\mathcal{H}}^{1,0})$ to $\mathcal{A}^0_{P(H^*)}(\wedge^{k-2}E)$. 
Hence there exists $Z_0 \in \mathcal{A}^0_{P(H^*)}(\wedge^{k-2}\widetilde{\mathcal{H}}^{1,0})$ satisfying (\ref{s5.3eq4}), and (iv). 
In order to find a solution $Y_0$ of the equation (ii) in Lemma~\ref{s5.3l2}, 
we consider the the cases $k\neq 1$ and $k=1$. 
In the case $k\neq 1$, we write (ii) as 
\begin{equation} 
\bar{\partial}(\widetilde{\theta}_0^{k-1}(Y_0))=k^{-2}c\omega_E(\widetilde{\theta}_0^k(X'_h), r^{-2}\widetilde{\theta}_1)
+r^2\widetilde{\theta}_0^{k-1}(Y_1)\otimes r^{-2}\overline{\eta}_1 
-(k-1)^2r^2\widetilde{\theta}_0^{k-2}(Z_0)\wedge_E r^{-2}\widetilde{\theta}_1. \label{s5.3eq5}
\end{equation}
The derivative $\bar{\partial}$ on the right hand side of (\ref{s5.3eq5}) is provided by 
\begin{equation}\label{s5.3eq8}
cr^{-2}\{ \wedge (\omega_E(\widetilde{\theta}_0^{k-1}(Y_1)\wedge_E \widetilde{\theta}_1, \widetilde{\theta}_1))
-2\widetilde{\theta}_0^{k-1}(Y_1)\otimes \omega_E(\widetilde{\theta}_1, \widetilde{\theta}_1)
-\omega_E(\widetilde{\theta}_0^{k-1}(Y_1), \widetilde{\theta}_1)\wedge \widetilde{\theta}_1\}. 
\end{equation}
We take a point $u^*$ of $P(H^*)$ and denote by $e$ the element $\widetilde{\theta}_0^{k-1}(Y_1)_{u^*}$. 
The $(0,1)$-form $\widetilde{\theta}_1$ is given by $\id_E\otimes u^*j$ at the point $u^*$. 
Then (\ref{s5.3eq8}) is written as an element 
\begin{equation}\label{s5.3eq9}
\wedge_{E^*} (\omega_E(e\wedge_E \id_E, \id_E))
-2e\otimes \omega_E(\id_E, \id_E) -\omega_E(e, \id_E)\wedge_{E, E^*} \id_E
\end{equation}
of $\wedge^{k-1}E\otimes \wedge^2 E^*$ by the basis $cr^{-2} u^*j \cdot u^*j$. 
However, the element (\ref{s5.3eq9}) vanishes by the calculation. 
Therefore (\ref{s5.3eq8}) vanishes at each point $u^*\in P(H^*)$. 
It turns out that the derivative $\bar{\partial}$ on the right hand side of (\ref{s5.3eq5}) is zero for $k\neq 1$. 
In the case $k=1$, by the same argument, the derivative $\bar{\partial}$ on the right hand side of (ii) in Lemma~\ref{s5.2l1} is reduced to 
$\wedge_{E^*} (\omega_E(\id_E, \id_E))-2\omega_E(\id_E, \id_E)$ which vanishes. 
Hence, there exists $Y_0\in \mathcal{A}^0_{P(H^*)}(\wedge^{k-1}\widetilde{\mathcal{H}})$ such that (\ref{s5.3eq5}) and (ii) hold for any $1\le k\le 2n$. 
It completes the proof. 
\end{proof}

Lemma~\ref{s5.3l2} and Proposition~\ref{s5.3p2} induce the following 
\begin{thm}\label{s5.3t1} 
Horizontal $k$ and $(k-1)$-vector fields $X'_h, Y_1$ satisfy for $1\le k\le 2n-1$, 
\[
\mathrm{(i)}\quad \bar{\partial}(\widetilde{\theta}_0^k(X'_h))-k^2\widetilde{\theta}_0^{k-1}(Y_1)\wedge_E \widetilde{\theta}_1=0
\]
and for $k=2n$, (i) and 
\[
\bar{\partial}(\widetilde{\theta}_0^{2n-1}(r^2Y_1))=0 
\]
if and only if the $(k,0)$-vector field $X'_h+Y_0\wedge v_0+Y_1\wedge v_1+ Z_0\wedge v_0\wedge v_1$ is holomorphic
for local horizontal $(k-1)$ and $(k-2)$-vector fields $Y_0, Z_0$ on $P(H^*)$. $\hfill\Box$
\end{thm}

\section{The twistor space $Z$}
The complex structure $\widetilde{I}$ on $P(H^*)$ induces a complex structure $\widehat{I}$ on $Z$ 
since the action of $\textrm{GL}(1,\mathbb{C})$ on $P(H^*)$ is holomorphic. 
Each fiber of the projection $f:Z\to M$ is isomorphic to $\mathbb{C}P^1$. 
We denote by $l$ a line bundle over $Z$ which is the hyperplane bundle on each fiber of $f$. 
If $H$ is not global, then $l$ is not also. However, $l^2$ is globally defined. 

\subsection{Lift of $\mathcal{A}^q(\wedge^kE \otimes S^mH)$ to $Z$}\label{s4.1}
Let $\mathcal{A}^0_Z(l^m)$ be a sheaf of smooth section of the $m$-th tensor product $l^m$ over $Z$. 
We define a sheaf $\widehat{\mathcal{A}}^0(l^m)$ by 
\[
\widehat{\mathcal{A}}^0(l^m)=\{\zeta \in f^{-1}f_*(\mathcal{A}^0_Z(l^m)) \mid \zeta : \text{holomorphic along each fiber of $f$}\}.
\]
In the case $m=0$, $\widehat{\mathcal{A}}^0(l^0)$ is just the sheaf of functions on $Z$ which are constant along each fiber of $f$. 
We write $\widehat{\mathcal{A}}^0(l^0)$ as $\widehat{\mathcal{A}}^0$ for simplicity. 
Let $\widehat{\mathcal{A}}^q(\wedge^k E)$ denote the sheaf of pull-back of $\wedge^k E$-valued $q$-forms on $M$ by $f$. 
We define a sheaf $\widehat{\mathcal{A}}^q(\wedge^k E\otimes l^m)$ as 
\begin{equation*}\label{s3.6eq1}
\widehat{\mathcal{A}}^q(\wedge^k E\otimes l^m)=\widehat{\mathcal{A}}^q(\wedge^k E)\otimes_{\widehat{\mathcal{A}}^0} \widehat{\mathcal{A}}^0(l^m).
\end{equation*}
Since a polynomial of degree $m$ on $\mathbb{C}^2\backslash \{0\}$ induces a holomorphic section of the $m$-th tensor product $l^m$, 
any element $\widetilde{\xi}_0$ of $\widetilde{\mathcal{A}}^q_{(m,0)}$ defines 
an element of $\widehat{\mathcal{A}}^q(l^m)$, which we denote by $\widehat{\xi}$. 
Such an element $\widehat{\xi}$ is called a \textit{lift of $\xi$ to $Z$}. 
The correspondence $\widetilde{\xi}_0\mapsto \widehat{\xi}$ provides the isomorphism 
\begin{equation}\label{s3.6eq2}
\widetilde{\mathcal{A}}^q_{(m,0)}(\wedge^k E)\cong \widehat{\mathcal{A}}^q(\wedge^k E\otimes l^m). 
\end{equation}
It follows from Corollary~\ref{s3.1c2} that  
\begin{cor}\label{s3.6c1} 
$\mathcal{A}^q(\wedge^k E\otimes S^mH)\cong \widehat{\mathcal{A}}^q(\wedge^k E\otimes l^m)$ by the correspondence $\xi\mapsto \widehat{\xi}$. $\hfill\Box$ 
\end{cor}

\subsection{Real structures on $Z$}
A differential $q$-form $\alpha$ on $P(H^*)$ is called of \textit{$\textrm{GL}(1,\mathbb{C})$-order $m$}
if $(R_c)_*\alpha=c^m\alpha$ for any $c\in \textrm{GL}(1,\mathbb{C})$. 
Let $\mathcal{A}^q_{P(H^*), m}(\wedge^k E)$ denote a subsheaf of $\mathcal{A}^q_{P(H^*)}(\wedge^k E)$ whose elements are of $\textrm{GL}(1,\mathbb{C})$-order $m$. 
The anti-$\mathbb{C}$ linear endomorphism $\widetilde{\tau}$ of $\mathcal{A}^q_{P(H^*)}(\wedge^k E)$ as in \S \ref{s3.2} induces that of $\mathcal{A}^q_{P(H^*), m}(\wedge^k E)$. 
If $k+m$ is even, $\widetilde{\tau}$ is a real structure of $\mathcal{A}^q_{P(H^*), m}(\wedge^k E)$. 

Let $\mathcal{A}^q_Z(f^{-1}(\wedge^kE)\otimes l^m)$ be a sheaf of $f^{-1}(\wedge^kE)\otimes l^m$-valued differential $q$-form on $Z$. 
We denote $\mathcal{A}^q_Z(\wedge^kE\otimes l^m)$ the sheaf $f^{-1}f_*\mathcal{A}^q_Z(f^{-1}(\wedge^kE)\otimes l^m)$. 
The sheaf $\mathcal{A}^q_Z(\wedge^kE\otimes l^m)$ is isomorphic to $\mathcal{A}^q_{P(H^*), m}(\wedge^k E)$. 
We define $\widehat{\tau}$ as the endomorphism of $\mathcal{A}^q_Z(\wedge^kE\otimes l^m)$ induced by $\widetilde{\tau}$. 
The right action $R_j$ of $j$ on $P(H^*)$ induces an anti-holomorphic involution of $Z$, and we denote it by $R_{[j]}:Z\to Z$. 
The map $R_{[j]}$ is the antipodal map of each fiber $\mathbb{C}P^1$ of $f$. 
The anti-$\mathbb{C}$ linear endomorphism $\widehat{\tau}$ of $\mathcal{A}^q_Z(\wedge^kE\otimes l^m)$ is given by 
\[
\widehat{\tau}(\beta_Z)=J_E^k\overline{R_{[j]}^*\beta_Z}
\]
for $\beta_Z \in \mathcal{A}^q_Z(\wedge^kE\otimes l^m)$. 
It follows from Proposition~\ref{s3.1.5p1} and the isomorphism (\ref{s3.6eq2}) that  
\begin{prop}\label{s3.8.5p1} 
The map $\widehat{\tau}$ defines an endomorphism of $\widehat{\mathcal{A}}^q(\wedge^kE\otimes l^m)$ 
such that $\widehat{\tau}(\widehat{\xi})=\widehat{\tau(\xi)}$ for $\xi \in \mathcal{A}^q(\wedge^k E \otimes S^mH)$. 
Moreover, $\widehat{\tau}$ is a real structure on $\widehat{\mathcal{A}}^q(\wedge^kE\otimes l^m)$ if $k+m$ is even. $\hfill\Box$
\end{prop}

Let $\widehat{\mathcal{A}}^q(\wedge^kE\otimes l^m)^{\widehat{\tau}}$ denote 
the sheaf of $\widehat{\tau}$-invariant elements of $\widehat{\mathcal{A}}^q(\wedge^kE\otimes l^m)$. 
Corollary~\ref{s3.1.5c1} and Proposition~\ref{s3.8.5p1} imply the following corollary : 
\begin{cor}\label{s3.8.5c1} 
$\mathcal{A}^q(\wedge^kE\otimes S^mH)^{\tau}\cong \widehat{\mathcal{A}}^q(\wedge^kE\otimes l^m)^{\widehat{\tau}}$ 
by $\xi\mapsto \widehat{\xi}$. $\hfill\Box$
\end{cor}


\subsection{Canonical 1-form on $Z$}
The principal $\textrm{GL}(1,\mathbb{C})$-bundle $\pi: P(H^*)\to Z$ is regarded as the frame bundle of $l^*$. 
An $l^m$-valued differential $(q, q')$-form on $Z$ is induced by a differential $(q, q')$-form on $P(H^*)$ of $\textrm{GL}(1,\mathbb{C})$-order $m$ which is annihilate to vectors along each fiber of $\pi$. 
We define $\widehat{\theta}_0$ and $\widehat{\theta}_1$ as the $f^{-1}(E)\otimes l$-valued $(1,0)$-form and the $f^{-1}(E)\otimes l^{-1}$-valued $(0,1)$-form on $Z$ 
induced by $\widetilde{\theta}_0$ and $r^{-2}\widetilde{\theta}_1$, respectively. 
Let $\eta$ and $\widehat{\omega}$ be the $l^2$-valued $(1,0)$-form and the $l^2$-valued $(2,0)$-form on $Z$ induced by $r^2\eta_1$ and $\widetilde{\omega}_0$, respectively. 
The forms $\widehat{\theta}_0$ and $\widehat{\omega}$ are the lift (as in the section \ref{s4.1}) of $\id \in \mathcal{A}^1(E\otimes H)$ and $\omega=\widehat{\omega}_E\otimes s^2_H \in \mathcal{A}^2(S^2H)$ to $Z$, respectively. 
The forms $\widehat{\theta}_0$, $\widehat{\theta}_1$, $\eta$ and $\widehat{\omega}$ are $\widehat{\tau}$-invariant 
since $\widetilde{\theta}_0$, $\widetilde{\theta}_1$, $\eta_1$, $\widetilde{\omega}_0$ and $r$ are $\widetilde{\tau}$-invariant. 

The line bundle $l$ admits a connection with the connection form $\eta_0$ on the frame bundle $P(H^*)$. 
Let $d^l$ be the covariant exterior derivative. 
Then $d^l$ corresponds to the restriction of $d^E$ to the horizontal $\ker \eta_0$. 
By Proposition \ref{s3.2p1} and \ref{s3.3p2}, we obtain 
\begin{prop}\label{s3.7p1} 
$d^l\widehat{\theta}_0 =-\eta \wedge \widehat{\theta}_1,\ d^l\eta =-2c \widehat{\omega}$. $\hfill\Box$
\end{prop}

It implies that $\eta$ is a holomorphic 1-form valued with $l^2$. 
\begin{prop}\label{s3.7p2} (c.f. Theorem~4.3 in \cite{S1}) 
If the scalar curvature $t$ is not zero, then $\eta$ is a holomorphic contact form on $Z$ such that $l^2$ is the contact bundle. 
\end{prop}
\begin{proof} 
If $t\neq 0$, then $c\neq 0$, and $(d^l\eta)^n\wedge \eta=(-2c)^n \widehat{\omega}^n\wedge \eta\neq 0$. 
\end{proof} 

The symmetric 2-tensor $\eta_1\otimes \overline{\eta}_1+\overline{\eta}_1\otimes \eta_1$ on $P(H^*)$ is $\textrm{GL}(1,\mathbb{C})$-invariant 
and annihilated by tangent vectors to the fiber of $\pi:P(H^*)\to Z$. 
Let $g_{\widehat{\mathcal{V}}}$ be a real symmetric 2-form on $Z$ such that $\pi^*g_{\widehat{\mathcal{V}}}=\eta_1\otimes \overline{\eta}_1+\overline{\eta}_1\otimes \eta_1$. 
We define a real symmetric 2-form $\widehat{g}$ by 
\[
\widehat{g}=cf^*g+g_{\widehat{\mathcal{V}}}
\]
on $Z$. 
\begin{prop}\label{s3.7p3}  (c.f. Theorem~6.1 in \cite{S1}) 
If $t$ is positive, 
then $(\widehat{g}, \widehat{I})$ is a K\"{a}hler-Einstein structure on $Z$ 
with positive scalar curvature. 
\end{prop}
\begin{proof} 
It follows from $\pi^*(\widehat{g}(\widehat{I}\cdot, \cdot))=-id\eta_0=i\partial \bar{\partial} \log r^2$ 
that $(\widehat{g}, \widehat{I})$ is a K\"{a}hler structure on $Z$. 
Proposition~\ref{s3.7p2} implies that the canonical bundle $K_Z$ of $Z$ is isomorphic to $l^{-2(n+1)}$. 
The Ricci form $\rho_{\widehat{g}}$ with respect to $\widehat{g}$ is given by the 2-form $-2(n+1)id\eta_0$ on $P(H^*)$. 
Hence, $\widehat{g}$ is a K\"{a}hler-Einstein metric on $Z$ of positive scalar curvature. 
\end{proof} 
As in the above proof, $\widehat{g}$ is the Fubini-Study metric on the fiber $\mathbb{C}P^1$ of $f:Z\to M$.

Let $\nabla$ be a torsion free connection on $Z$ such that $\nabla^{0,1}=\bar{\partial}$. 
It follows from Proposition \ref{s3.7p1} that $\nabla^{0,1} \widehat{\theta}_0=\eta\otimes \widehat{\theta}_1$ and $\nabla^{0,1} \eta=0$. 
We define a $f^{-1}(\wedge^k E) \otimes l^k$-valued $(k,0)$-form $\widehat{\theta}_0^k$ by the $k$-th wedge of $\widehat{\theta}_0$. 
Then 
\begin{prop}\label{s3.7p4} 
$\nabla^{0,1} \widehat{\theta}_0^k=k\widehat{\theta}_0^{k-1}\wedge \eta \wedge_E \widehat{\theta}_1,\ 
\nabla^{0,1} (\widehat{\theta}_0^{k-1}\wedge \eta)=0$. $\hfill\Box$ 
\end{prop}

\subsection{Holomorphic $k$-vector fields on $Z$}
The horizontal bundle $\widetilde{\mathcal{H}}$ induces a bundle $\widehat{\mathcal{H}}$ over the twistor space $Z$ 
since $\widetilde{\mathcal{H}}$ is invariant under the action of $\textrm{GL}(1,\mathbb{C})$. 
Let $\widehat{\mathcal{V}}$ be the bundle of tangent vectors of each fiber of $f:Z\to M$. 
The tangent bundle $TZ\otimes \mathbb{C}$ is decomposed into $TZ\otimes \mathbb{C}=\widehat{\mathcal{H}}\oplus \widehat{\mathcal{V}}$. 
The bundle $\widehat{\mathcal{H}}$ is isomorphic to the pull back bundle $f^{-1}(TM\otimes \mathbb{C})$. 
We call a section of $\wedge^k \widehat{\mathcal{H}}$ a \textit{horizontal $k$-vector field on $Z$}. 
Let $\widehat{\mathcal{H}}^{1,0}$ be a vector bundle of horizontal $(1,0)$-vectors on $Z$. 
The bundle $\widehat{\mathcal{H}}^{1,0}$ is a holomorphic subbundle of $T^{1,0}Z$ since $\widehat{\mathcal{H}}^{1,0}$ is the kernel of the holomorphic form $\eta$. 
The $1$-form $\eta$ provides the map from $\wedge ^k T^{1,0}Z$ to $l^2\otimes \wedge ^{k-1} T^{1,0}Z$, which we also denote by $\eta$. 
Then $\wedge^k \widehat{\mathcal{H}}^{1,0}$ is the kernel of the map $\eta$. 
We denote by $v$ the $l^{-2}$-valued (1,0)-vector field on $Z$ induced by the vector field $r^{-2}v_1$ on $P(H^*)$. 
The $l^{-2}$-valued vector field $v$ is regarded as the dual of $\eta$ since $\eta (v)=1$. 
We remark that $v$ is not holomorphic on whole $Z$ but holomorphic along each fiber. 
Let $X'$ be a $(k,0)$-vector field on $Z$. 
Then $X'$ is given by 
\[
X'=X'_h+Y \wedge v
\]
for $X'_h\in \wedge^k\widehat{\mathcal{H}}^{1,0}$ and $Y\in l^2\otimes \wedge^{k-1}\widehat{\mathcal{H}}^{1,0}$. 
\begin{thm}\label{s5.6t1}
For $1\le k\le 2n-1$, the $(k,0)$-vector field $X'_h+Y \wedge v$ is holomorphic if and only if 
\[
\bar{\partial}^l(\widehat{\theta}_0^k(X'_h))-k^2\widehat{\theta}_0^{k-1}(Y)\wedge_E \widehat{\theta}_1=0.
\]
The $(2n,0)$-vector field $X'_h+Y \wedge v$ is holomorphic if and only if 
\begin{align*}
&\bar{\partial}^l(\widehat{\theta}_0^{2n}(X'_h))-4n^2\widehat{\theta}_0^{2n-1}(Y)\wedge_E \widehat{\theta}_1=0, \\
&\bar{\partial}^l(\widehat{\theta}_0^{2n-1}(Y))=0.
\end{align*}
\end{thm}
\begin{proof} 
The $(k,0)$-vector field $X'$ is holomorphic if and only if $\nabla^{0,1}X'=0$. 
The equation is equal to $\widehat{\theta}_0^k(\nabla^{0,1}X')=0$ and 
$(\widehat{\theta}_0^{k-1}\wedge \eta)(\nabla^{0,1}X')=0$. 
Proposition~\ref{s3.7p4} implies that 
$\widehat{\theta}_0^k(\nabla^{0,1}X') =\bar{\partial}^l(\widehat{\theta}_0^k(X'_h)) - k^2\widehat{\theta}_0^{k-1}(Y)\wedge_E \widehat{\theta}_1$. 
Now $\overline{\partial}^{\nabla}(\widehat{\theta}_0^k(\nabla^{0,1}X'))=(-1)^{k+1}(\widehat{\theta}_0^{k-1}\wedge \eta)(\nabla^{0,1}X')\wedge \widehat{\theta}_1$. 
In the case $1\le k \le 2n-1$, 
the map $\wedge \widehat{\theta}_1 : f^{-1}(\wedge^{k-1}E) \otimes l^{k+1}\otimes \wedge^{0,1} \to f^{-1}(\wedge^k E) \otimes l^k \otimes\wedge^{0,2}$ is injective, 
and the equation $(\widehat{\theta}_0^{k-1}\wedge \eta)(\nabla^{0,1}X')=0$ follows from $\widehat{\theta}_0^k(\nabla^{0,1}X')=0$. 
In the case $k = 2n$, 
$(\widehat{\theta}_0^{2n-1}\wedge \eta)(\nabla^{0,1}X')=2n\bar{\partial}^l(\widehat{\theta}_0^{2n-1}(Y))$. 
Hence we finish the proof. 
\end{proof}

\section{Quaternionic sections}
In this section, we provide a definition of a quaternionic section of $\wedge^k E\otimes S^mH$. 
We show that the lifts of the quaternionic section satisfy some $\bar{\partial}$-equations on $P(H^*)$ and $Z$. 

\subsection{Quaternionic sections of $\wedge^k E\otimes S^mH$}
The Levi-Civita connection of $(M,g)$ induces the covariant derivative $\nabla :\Gamma(\wedge^k E\otimes S^mH) \to \Gamma(\wedge^k E\otimes  S^mH \otimes E^*\otimes H^*)$ 
where $\Gamma(\wedge^k E\otimes S^mH)$ means the space of smooth sections of $\wedge^k E\otimes S^mH$. 
The space $\wedge^k E\otimes S^mH \otimes E^*\otimes H^*$ is isomorphic to $\wedge^k E\otimes  E^* \otimes S^mH\otimes H$ 
by the isomorphism $\omega_{H^*}^{\sharp} : H^* \to H$. 
Moreover, $S^mH\otimes H\cong S^{m+1}H\oplus S^{m-1}H$ by the Clebsch-Gordan decomposition. 
Thus, the covariant derivative $\nabla$ is regarded as 
\[
\nabla : \Gamma(\wedge^k E\otimes S^mH) \to \Gamma(\wedge^k E\otimes E^* \otimes S^{m+1}H) \oplus \Gamma(\wedge^k E\otimes E^*\otimes S^{m-1}H). 
\]
Dirac operator is defined as as the $\wedge^k E\otimes E^*\otimes S^{m+1}H$-part of $\nabla$ (c.f.~\cite{B}) : 
\[
\mathfrak{D}_{\wedge^k E}: \Gamma(\wedge^k E\otimes S^mH) \to \Gamma(\wedge^k E\otimes E^*\otimes S^{m+1}H)
\]
Let $k$ be a positive integer. 
By the restriction of $\otimes^k E\otimes E^*$ to $\wedge^k E\otimes E^*$, 
the trace of $(\otimes^k E) \otimes E^*$ induces the map $\tr : \wedge^k E \otimes E^*\to \wedge^{k-1} E$. 
Let $(\wedge^k E\otimes E^*)_0$ denote the kernel of $\tr: \wedge^k E \otimes E^*\to \wedge^{k-1} E$. 
We rescale the trace map as $\frac{1}{2n-k+1}\tr$, and also denote it by the same notation $\tr$. 
Then the map $\tr : \wedge^k E\otimes E^* \to \wedge^{k-1} E$ have a right inverse $\alpha \mapsto \alpha \wedge \id_E$ 
where $\alpha \wedge \id_E$ is the image of $\alpha \otimes \id_E$ 
by the anti-symmetrization $\wedge^{k-1}E\otimes E\otimes E^*  \to \wedge^k E\otimes E^*$. 
Hence, the bundle $\wedge^k E\otimes E^*$ is decomposed into $(\wedge^k E\otimes E^*)_0$ and $(\wedge^{k-1} E)\wedge \id_E$ : 
\[
\wedge^k E\otimes E^*=(\wedge^k E\otimes E^*)_0 \oplus (\wedge^{k-1} E)\wedge \id_E. 
\]
We define an operator 
\[
\mathfrak{D}_{\wedge^k E}^0: \Gamma(\wedge^k E\otimes S^mH) \to \Gamma((\wedge^k E\otimes E^*)_0\otimes S^{m+1}H)
\]
as the $(\wedge^k E\otimes E^*)_0$-part of $\mathfrak{D}_{\wedge^k E}$. 
\begin{defi}\label{s4.1d1}
{\rm 
Let $m$ be a non-negative integer. 
A section $X$ of $\wedge^k E\otimes S^mH$ is \textit{quaternionic} if $\mathfrak{D}_{\wedge^k E}^0(X)=0$ for $1\le k\le 2n-1$ and 
$\mathfrak{D}_{\wedge^{2n-1} E}\circ \tr\circ \mathfrak{D}_{\wedge^{2n} E}(X)=0$ for $k=2n$. 
}
\end{defi}
Any section $X$ of $\wedge^{2n} E\otimes S^mH$ satisfies $\mathfrak{D}_{\wedge^{2n} E}^0(X)=0$ 
since $(\wedge^{2n} E\otimes E^*)_0= \{0\}$. 

Let $\tau$ be an anti-$\mathbb{C}$-linear endomorphism of $\Gamma((\otimes^kE)\otimes (\otimes^{k'} E^*)\otimes (\otimes^m H)\otimes (\otimes^{m'} H^*)\otimes \wedge^qT^*M)$ as 
\[
\tau(\xi)=(J_E^k \otimes J_{E^*}^{k'} \otimes J_H^m\otimes J_{H^*}^{m'})(v_i)\otimes \overline{\alpha^i}
\]
for $\xi=\sum_i v_i\otimes \alpha^i$ where $\{v_i\}$ is a frame of $(\otimes^kE)\otimes (\otimes^{k'} E^*)\otimes (\otimes^m H)\otimes (\otimes^{m'} H^*)$ and 
$\alpha^i$ is a $q$-form. 
If $k+k'+m+m'$ is even, then $\tau$ is a real structure. 
The covariant derivative $\nabla : \Gamma(\wedge^k E\otimes S^mH) \to \Gamma(\wedge^k E\otimes S^mH\otimes T^*M)$ satisfies $\tau\circ \nabla=\nabla\circ \tau$ 
since the connections of $E$ and $H$ preserve $J_E$ and $J_H$, respectively. 
The maps $\omega_{H^*}^{\sharp} : \Gamma(E^*\otimes H^*) \to \Gamma(E^*\otimes H)$ and 
$s_{H}^{m+1}: \Gamma(\wedge^kE\otimes E^* \otimes S^mH\otimes H) \to \Gamma(\wedge^kE\otimes E^* \otimes S^{m+1}H)$ 
are commutative with $\tau$. 
Since $\mathfrak{D}_{\wedge^k E}$ is provided by $s_{H}^{m+1}\circ \omega_{H^*}^{\sharp}\circ \nabla$, 
the operator $\mathfrak{D}_{\wedge^k E}$ is also commutative with $\tau$. 
The trace map $\tr :\wedge^k E\otimes E^*\to \wedge^{k-1} E$ satisfies $\tr \circ (J_E^k\otimes J_{E^*})=J_E^{k-1}\circ \tr$. 
Hence, the operators $\mathfrak{D}_{\wedge^k E}^0$ and $\mathfrak{D}_{\wedge^{2n-1} E}\circ \tr\circ \mathfrak{D}_{\wedge^{2n} E}$ are commutative with $\tau$. 


\subsection{Lift of quaternionic sections to $P(H^*)$}
The map $\omega_{H^*}^{\sharp}: H^* \to H$ induces 
the isomorphism from $\mathcal{A}^1=\mathcal{A}^0(E^*\otimes H^*)$ to $\mathcal{A}^0(E^*\otimes H)$, 
and denote it also by $\omega_{H^*}^{\sharp}$. 
On $P(H^*)$, $\omega_{H^*}^{\sharp}$ provides a map from $\mathcal{A}^0_{P(H^*)}(\widetilde{\mathcal{H}}^*)$ to $\mathcal{A}^0_{P(H^*)}(E^*\otimes H)$ 
since $\widetilde{\mathcal{H}}^*_{u^*}\cong p^{-1}(E^*\otimes H^*)_{u^*}$ at $u^*\in P(H^*)$. 
The sheaf $\mathcal{A}^0_{P(H^*)}(E^*\otimes H)$ is isomorphic to $\mathcal{A}^0_{P(H^*)}(E^*\otimes \mathbb{H})$ by considering a point $u^*\in P(H^*)$ as a frame of $H^*$. 
Moreover, taking a coefficient of $1\in \mathbb{H}$, we have a map $\mathcal{A}^0_{P(H^*)}(E^*\otimes \mathbb{H}) \to \mathcal{A}^0_{P(H^*)}(E^*)$. 
Therefore, we obtain a map $\mathcal{A}^0_{P(H^*)}(\widetilde{\mathcal{H}}^*)\to \mathcal{A}^0_{P(H^*)}(E^*)$ and denote it by $\widetilde{\omega}_{H^*}^{\sharp}$. 
The $(1,0)$ and $(0,1)$-subspaces $(\widetilde{\mathcal{H}}^*)^{1,0}$ and $(\widetilde{\mathcal{H}}^*)^{0,1}$ of $\widetilde{\mathcal{H}}^*$ are given by 
$E^*\otimes u^*$ and $E^*\otimes u^*j$ at $u^*\in P(H^*)$. 
Since $\omega_{H^*}^{\sharp}(u^*)=-r^2(-u^*j)^*=-r^2ju$ and $\omega_{H^*}^{\sharp}(-u^*j)=r^2(u^*)^*=r^2u$, 
the map $\widetilde{\omega}_{H^*}^{\sharp}: \mathcal{A}^0_{P(H^*)}(\widetilde{\mathcal{H}}^*)\to \mathcal{A}^0_{P(H^*)}(E^*)$ has the kernel $\mathcal{A}^0_{P(H^*)}((\widetilde{\mathcal{H}}^*)^{1,0})$ 
and it is injective on $\mathcal{A}^0_{P(H^*)}((\widetilde{\mathcal{H}}^*)^{0,1})$. 
The restriction of $\widetilde{\omega}_{H^*}^{\sharp}$ to $\widetilde{\mathcal{A}}^1$ is 
the map $\widetilde{\omega}_{H^*}^{\sharp}: \widetilde{\mathcal{A}}^1 \to  \widetilde{\mathcal{A}}^0_{(1,0)}(E^*)$. 
Then $\widetilde{\omega}_{H^*}^{\sharp}$ is regarded as a lift of $\omega_{H^*}^{\sharp}: \mathcal{A}^1\to \mathcal{A}^0(E^*\otimes H)$ to $P(H^*)$. 
By the tensor product of the lift and $\widetilde{\mathcal{A}}^0_{(m,0)}$, we obtain the map 
$\widetilde{\omega}_{H^*}^{\sharp}: \widetilde{\mathcal{A}}^1_{(m,0)} \to \widetilde{\mathcal{A}}^0(E^*)\otimes \widetilde{\mathcal{A}}^0_{(1,0)}\otimes \widetilde{\mathcal{A}}^0_{(m,0)}$
which corresponds to the map $\omega_{H^*}^{\sharp}: \mathcal{A}^1(S^mH)\to \mathcal{A}^0(E^*\otimes H\otimes S^mH)$. 
It follows from $\widetilde{\mathcal{A}}^0_{(1,0)}\otimes \widetilde{\mathcal{A}}^0_{(m,0)}=\widetilde{\mathcal{A}}^0_{(m+1,0)}$ that 
there exists a commutative diagram 
\begin{equation*}\label{s4.4eq3}
\begin{array}{ccccc}
\widetilde{\mathcal{A}}^1_{(m,0)}& \stackrel{\widetilde{\omega}_{H^*}^{\sharp}}{\longrightarrow} 
&\widetilde{\mathcal{A}}^0(E^*)\otimes \widetilde{\mathcal{A}}^0_{(1,0)} \otimes \widetilde{\mathcal{A}}^0_{(m,0)}
& = & \widetilde{\mathcal{A}}^0_{(m+1,0)}(E^*)  \\
\uparrow & & \uparrow & & \uparrow \\
\mathcal{A}^1(S^mH)&\stackrel{\omega_{H^*}^{\sharp}}{\longrightarrow } &\mathcal{A}^0(E^*\otimes H\otimes S^mH)
&\stackrel{s^{m+1}_{H}}{\longrightarrow } &\mathcal{A}^0(E^*\otimes S^{m+1}H). 
\end{array} 
\end{equation*}
By extending the above diagram to $\wedge^kE$-valued $1$-forms, we have 
\begin{equation*}\label{s4.4eq3.5}
\begin{array}{ccccc}
\widetilde{\mathcal{A}}^1_{(m,0)}(\wedge^k E)& \stackrel{\widetilde{\omega}_{H^*}^{\sharp}}{\longrightarrow} & \widetilde{\mathcal{A}}^0_{(m+1,0)}(\wedge^k E\otimes E^*)  \\
\uparrow & & \uparrow  \\
\mathcal{A}^1(\wedge^kE\otimes S^mH)&\stackrel{s^{m+1}_{H}\circ \omega_{H^*}^{\sharp}}{\longrightarrow } &\mathcal{A}^0(\wedge^k E\otimes E^*\otimes S^{m+1}H).
\end{array} 
\end{equation*}

The $\widetilde{\mathcal{H}}^*$-part $d_{\widetilde{\mathcal{H}}}: \mathcal{A}^0_{P(H^*)}(\wedge^k E)\to \mathcal{A}^0_{P(H^*)}(\widetilde{\mathcal{H}}^*\otimes \wedge^k E)$ of the exterior derivative 
induces $d_{\widetilde{\mathcal{H}}}: \widetilde{\mathcal{A}}^0_{(m,0)}(\wedge^k E)\to \widetilde{\mathcal{A}}^1_{(m,0)}(\wedge^k E)$ 
on the subsheaf $\widetilde{\mathcal{A}}^0_{(m,0)}(\wedge^k E)$ of $\mathcal{A}^0_{P(H^*)}(\wedge^k E)$. 
By Corollary~\ref{s3.1c2}, we obtain a commutative diagram 
\begin{equation}\label{s4.4eq4}
\begin{array}{ccccc}
\widetilde{\mathcal{A}}^0_{(m,0)}(\wedge^k E)& \stackrel{d_{\widetilde{\mathcal{H}}}}{\longrightarrow} 
&\widetilde{\mathcal{A}}^1_{(m,0)}(\wedge^k E)& \stackrel{\widetilde{\omega}_{H^*}^{\sharp}}{\longrightarrow} & \widetilde{\mathcal{A}}^0_{(m+1,0)}(\wedge^k E\otimes E^*)  \\
\uparrow & & \uparrow & & \uparrow  \\
\mathcal{A}^0(\wedge^kE\otimes S^mH)&\stackrel{\nabla}{\longrightarrow } 
&\mathcal{A}^1(\wedge^kE\otimes S^mH)&\stackrel{s^{m+1}_{H}\circ \omega_{H^*}^{\sharp}}{\longrightarrow } &\mathcal{A}^0(\wedge^k E\otimes E^*\otimes S^{m+1}H).
\end{array} 
\end{equation}
The operator $d_{\widetilde{\mathcal{H}}}:\widetilde{\mathcal{A}}^0_{(m,0)}(\wedge^k E)\to\widetilde{\mathcal{A}}^1_{(m,0)}(\wedge^k E)$ is decomposed into 
two operators $\partial_{\mathcal{H}}:\widetilde{\mathcal{A}}^0_{(m,0)}(\wedge^k E)\to \mathcal{A}^0_{P(H^*)}((\widetilde{\mathcal{H}}^*)^{1,0}\otimes \wedge^k E)$ 
and $\bar{\partial}_{\widetilde{\mathcal{H}}}:\widetilde{\mathcal{A}}^0_{(m,0)}(\wedge^k E)\to \mathcal{A}^0_{P(H^*)}((\widetilde{\mathcal{H}}^*)^{0,1}\otimes \wedge^k E)$ 
by the decomposition $\widetilde{\mathcal{H}}^*=(\widetilde{\mathcal{H}}^*)^{1,0}\oplus (\widetilde{\mathcal{H}}^*)^{0,1}$. 
The operator $\bar{\partial}_{\widetilde{\mathcal{H}}}$ coincides with $\bar{\partial}$ 
since an element of $\widetilde{\mathcal{A}}^0_{(m,0)}(\wedge^k E)$ is holomorphic along each fiber. 
\begin{prop}\label{s4.4p1} 
$(\widetilde{\mathfrak{D}_{\wedge^kE}\xi})_0=\widetilde{\omega}_{H^*}^{\sharp}(\bar{\partial} \widetilde{\xi}_0)$ for $\xi\in \mathcal{A}^0(\wedge^kE\otimes S^mH)$. 
\end{prop}
\begin{proof} 
It follows from $\mathfrak{D}_{\wedge^kE}=s^{m+1}_{H}\circ \omega_{H^*}^{\sharp}\circ \nabla$ and the diagram (\ref{s4.4eq4}) 
that $(\widetilde{\mathfrak{D}_{\wedge^kE}\xi})_0=\widetilde{\omega}_{H^*}^{\sharp}(d_{\widetilde{\mathcal{H}}} \widetilde{\xi}_0)$ 
for $\xi\in \mathcal{A}^0(\wedge^kE\otimes S^mH)$. 
Since the kernel of $\widetilde{\omega}_{H^*}^{\sharp}$ is $\mathcal{A}^0_{P(H^*)}((\widetilde{\mathcal{H}}^*)^{1,0}\otimes \wedge^k E)$, 
$\widetilde{\omega}_{H^*}^{\sharp}(d_{\widetilde{\mathcal{H}}} \widetilde{\xi}_0)=\widetilde{\omega}_{H^*}^{\sharp}(\bar{\partial} \widetilde{\xi}_0)$. 
Hence $(\widetilde{\mathfrak{D}_{\wedge^kE}\xi})_0=\widetilde{\omega}_{H^*}^{\sharp}(\bar{\partial}\widetilde{\xi}_0)$. 
\end{proof}
We denote by $\widetilde{\mathcal{O}}_m(\wedge^k E)$ the kernel of $\bar{\partial}$ on $\widetilde{\mathcal{A}}^0_{(m,0)}(\wedge^k E)$. 
By Proposition~\ref{s4.4p1} and the injectivity of $\widetilde{\omega}_{H^*}^{\sharp}$ on $\mathcal{A}^0_{P(H^*)}((\widetilde{\mathcal{H}}^*)^{0,1}\otimes \wedge^k E)$, 
we obtain an isomorphism 
\[
\Ker \mathfrak{D}_{\wedge^kE}\cong \widetilde{\mathcal{O}}_{m}(\wedge^k E) 
\]
by $\xi \mapsto \widetilde{\xi}_0$. 
We extend Proposition~\ref{s4.4p1} to the following : 
\begin{prop}\label{s4.5p1} 
$\widetilde{(\mathfrak{D}_{\wedge^kE}\xi-\zeta\wedge_E \id_E)}_0
=\widetilde{\omega}_{H^*}^{\sharp}(\bar{\partial}\widetilde{\xi}_0-\widetilde{\zeta}_0\wedge_E r^{-2}\widetilde{\theta}_1)$ 
for $\xi\in \mathcal{A}^0(\wedge^kE\otimes S^mH)$ and $\zeta \in \mathcal{A}^0(\wedge^{k-1} E\otimes S^{m+1}H)$. 
\end{prop}
\begin{proof} 
The canonical form $\widetilde{\theta}$ is given by $\widetilde{\theta}_{u^*}=\widetilde{\id}_E\otimes u^*\otimes 1+\widetilde{\id}_E\otimes (ju)^*\otimes j$ at $u^*\in P(H^*)$. 
Thus $(\widetilde{\theta}_1)_{u^*}=\widetilde{\id}_E\otimes (ju)^*=\widetilde{\id}_E\otimes u^*(-j)$. 
Now $\widetilde{\omega}_{H^*}^{\sharp}(\widetilde{\theta}_1)=r^2\widetilde{\id}_E$ since $\omega_{H^*}^{\sharp}(u^*(-j))=r^2u$. 
Hence we obtain 
\begin{equation}\label{s4.5eq12}
\widetilde{\omega}_{H^*}^{\sharp}(\widetilde{\zeta}_0\wedge_E r^{-2}\widetilde{\theta}_1)
=\widetilde{\zeta}_0\wedge_E \widetilde{\id}_E=(\widetilde{\zeta\wedge_E \id_E})_0. 
\end{equation}
The equation (\ref{s4.5eq12}) and Proposition~\ref{s4.4p1} imply the equation of this Proposition. 
\end{proof}

\begin{prop}\label{s4.5p2} 
Let $\xi$ and $\zeta$ be elements of $\mathcal{A}^0(\wedge^k E \otimes S^mH)$ and $\mathcal{A}^0(\wedge^{k-1} E \otimes S^{m+1}H)$, respectively.  
The element $\xi$ is quaternionic and $\zeta=\tr\circ\mathfrak{D}_{\wedge^kE}(\xi)$ 
if and only if $\bar{\partial}\widetilde{\xi}_0-\widetilde{\zeta}_0\wedge_E r^{-2}\widetilde{\theta}_1=0$ for $1\le k\le 2n-1$, 
and $\bar{\partial}\widetilde{\xi}_0-\widetilde{\zeta}_0\wedge_E r^{-2}\widetilde{\theta}_1=0$, $\bar{\partial}\widetilde{\zeta}_0=0$ for $k=2n$. 
\end{prop}
\begin{proof} 
Since $\tr(\zeta\wedge_E \id_E)=\zeta$ for any $\zeta \in \mathcal{A}^0(\wedge^{k-1} E\otimes S^{m+1}H)$, 
$\xi\in \mathcal{A}^0(\wedge^k E \otimes S^mH)$ is quaternionic and $\zeta=\tr\circ\mathfrak{D}_{\wedge^kE}(\xi)$ 
if and only if $\mathfrak{D}_{\wedge^kE}\xi-\zeta\wedge_E \id_E=0$ and, in addition, $\mathfrak{D}_{\wedge^{k-1}E}\zeta=0$ for $k=2n$. 
By Proposition~\ref{s4.5p1} and the injectivity of $\widetilde{\omega}_{H^*}^{\sharp}$ on $(\widetilde{\mathcal{H}}^*)^{0,1}$, 
$\mathfrak{D}_{\wedge^kE}\xi-\zeta\wedge_E \id_E=0$ is equal to $\bar{\partial}\widetilde{\xi}_0-\widetilde{\zeta}_0\wedge_E r^{-2}\widetilde{\theta}_1=0$. 
Furthermore, Proposition~\ref{s4.4p1} implies that $\mathfrak{D}_{\wedge^{k-1}E}\zeta=0$ is equivalent to $\bar{\partial}\widetilde{\zeta}_0=0$. 
\end{proof}

\subsection{Lift of quaternionic sections to $Z$} 
The sheaf $\mathcal{A}^0_{P(H^*)}(\widetilde{\mathcal{H}}^*)$ is considered as that of 1-forms which are annihilate to vertical vectors of $p:P(H^*)\to M$. 
The map $\widetilde{\omega}_{H^*}^{\sharp}$ maps an element of $\mathcal{A}^0_{P(H^*)}(\widetilde{\mathcal{H}}^*)$ of $\textrm{GL}(1,\mathbb{C})$-order $m$ 
to an element of $\mathcal{A}^0_{P(H^*), m+1}(E^*)$. 
Thus $\widetilde{\omega}_{H^*}^{\sharp}$ induces a map $\widehat{\omega}_{H^*}^{\sharp}: \mathcal{A}^0_{Z}(\widehat{\mathcal{H}}^*\otimes l^m)\to \mathcal{A}^0_{Z}(E^*\otimes l^{m+1})$ on $Z$. 
Since the kernel of $\widehat{\omega}_{H^*}^{\sharp}$ is $\mathcal{A}^0_{Z}((\widehat{\mathcal{H}}^*)^{1,0}\otimes l^m)$, 
$\widehat{\omega}_{H^*}^{\sharp}$ is injective on $\mathcal{A}^0_{Z}((\widehat{\mathcal{H}}^*)^{0,1}\otimes l^m)$. 
We restrict $\widehat{\omega}_{H^*}^{\sharp}$ to $\widehat{\mathcal{A}}^1(l^m)$ and 
and extend the map to $\widehat{\omega}_{H^*}^{\sharp}: \widehat{\mathcal{A}}^1(\wedge^kE\otimes l^m) \to \widehat{\mathcal{A}}^0(\wedge^kE\otimes E^*\otimes l^{m+1})$. 
The derivative $d^l$ induces an operator $d_{\widehat{\mathcal{H}}}^l: \widehat{\mathcal{A}}^0(\wedge^k E\otimes l^m)\to \widehat{\mathcal{A}}^1(\wedge^k E\otimes l^m)$ 
by the restriction to $\widehat{\mathcal{H}}$. 
Then there exists a commutative diagram 
\begin{equation*}\label{s4.6eq2} 
\begin{array}{ccccc}
\widehat{\mathcal{A}}^0(\wedge^kE\otimes l^m)&\stackrel{d_{\widehat{\mathcal{H}}}^l}{\longrightarrow}&
\widehat{\mathcal{A}}^1(\wedge^kE\otimes l^m)& \stackrel{\widehat{\omega}_{H^*}^{\sharp}}{\longrightarrow} & \widehat{\mathcal{A}}^0(\wedge^kE\otimes E^*\otimes l^{m+1}) \\
\uparrow &&\uparrow & & \uparrow  \\
\widetilde{\mathcal{A}}^0_{(m,0)}(\wedge^k E)&\stackrel{d_{\widetilde{\mathcal{H}}}}{\longrightarrow}&
\widetilde{\mathcal{A}}^1_{(m,0)}(\wedge^k E)& \stackrel{\widetilde{\omega}_{H^*}^{\sharp}}{\longrightarrow} & \widetilde{\mathcal{A}}^0_{(m+1,0)}(\wedge^kE\otimes E^*)  \\
\uparrow &&\uparrow & & \uparrow  \\
\mathcal{A}^0(\wedge^k E\otimes S^mH)&\stackrel{\nabla}{\longrightarrow }&\mathcal{A}^1(\wedge^k E\otimes S^mH)
&\stackrel{s^{m+1}_{H}\circ \omega_{H^*}^{\sharp}}{\longrightarrow } &\mathcal{A}^0(\wedge^kE\otimes E^*\otimes S^{m+1}H).
\end{array} 
\end{equation*}
The operator $d_{\widehat{\mathcal{H}}}^l$ is decomposed into two operators 
$\partial_{\widehat{\mathcal{H}}}^l : \widehat{\mathcal{A}}^0(\wedge^kE\otimes l^m)\to \mathcal{A}^0_{Z}((\widehat{\mathcal{H}}^*)^{1,0}\otimes \wedge^k E\otimes l^m)$ 
and $\bar{\partial}_{\widehat{\mathcal{H}}}^l : \widehat{\mathcal{A}}^0(\wedge^kE\otimes l^m)\to \mathcal{A}^0_{Z}((\widehat{\mathcal{H}}^*)^{0,1}\otimes \wedge^k E\otimes l^m)$ 
by $\widehat{\mathcal{H}}^*=(\widehat{\mathcal{H}}^*)^{1,0}\oplus (\widehat{\mathcal{H}}^*)^{0,1}$. 
Then $\bar{\partial}_{\widehat{\mathcal{H}}}^l$ is just the operator $\bar{\partial}^l$ 
since an element of $\widehat{\mathcal{A}}^0(\wedge^kE\otimes l^m)$ is holomorphic along each fiber of $f:Z\to M$. 
By the same proof of Proposition~\ref{s4.4p1}, we obtain 
\begin{prop}\label{s4.6p1}
$\widehat{\mathfrak{D}_{\wedge^kE}\xi} =\widehat{\omega}_{H^*}^{\sharp}(\bar{\partial}^l \widehat{\xi})$ for $\xi\in \mathcal{A}^0(\wedge^kE\otimes S^mH)$. $\hfill\Box$
\end{prop}

We denote by $\widehat{\mathcal{O}}(\wedge^kE\otimes l^m)$ the kernel of $\bar{\partial}^l$ on $\widehat{\mathcal{A}}^0(\wedge^kE\otimes l^m)$. 
Then $\widehat{\mathcal{O}}(\wedge^kE\otimes l^m)$ is a subsheaf of $\widehat{\mathcal{A}}^0(\wedge^kE\otimes l^m)$, 
whose elements are holomorphic sections of $f^{-1}(\wedge^kE)\otimes l^m$. 
Proposition~\ref{s4.6p1} and the injectivity of $\widehat{\omega}_{H^*}^{\sharp}$ on $(\widehat{\mathcal{H}}^*)^{0,1}$ induce an isomorphism 
\begin{equation}\label{s4.6e3} 
\Ker \mathfrak{D}_{\wedge^kE}\cong \widehat{\mathcal{O}}(\wedge^k E\otimes l^m)
\end{equation}
by $\xi \mapsto \widehat{\xi}$. 
We extend Proposition~\ref{s4.6p1} to the following proposition : 
\begin{prop}\label{s4.7p1}
$\widehat{\mathfrak{D}_{\wedge^kE}\xi-\zeta\wedge_E \id_E}=\widehat{\omega}_{H^*}^{\sharp}(\bar{\partial}^l \widehat{\xi}-\widehat{\zeta} \wedge_E \widehat{\theta}_1)$. 
\end{prop}
\begin{proof}
The terms $\widetilde{\omega}_{H^*}^{\sharp}(\widetilde{\zeta}_0\wedge_E r^{-2}\widetilde{\theta}_1)$ and $(\widetilde{\zeta\wedge_E \id_E})_0$ in (\ref{s4.5eq12}) 
correspond to $\widehat{\omega}_{H^*}^{\sharp}(\widehat{\zeta} \wedge_E \widehat{\theta}_1)$ and $\widehat{\zeta \wedge_E \id_E}$, respectively. 
Then $\widehat{\omega}_{H^*}^{\sharp}(\widehat{\zeta} \wedge_E \widehat{\theta}_1) =\widehat{\zeta \wedge_E \id_E}$. 
The equation and Proposition~\ref{s4.6p1} induce this Proposition. 
\end{proof}

As the proof of Proposition~\ref{s4.5p2}, Proposition~\ref{s4.6p1} and~\ref{s4.7p1} imply the following : 
\begin{prop}\label{s4.7p2} 
Let $\xi$ and $\zeta$ be elements of $\mathcal{A}^0(\wedge^k E \otimes S^mH)$ and $\mathcal{A}^0(\wedge^{k-1} E \otimes S^{m+1}H)$, respectively.  
The element $\xi$ is quaternionic and $\zeta=\tr\circ\mathfrak{D}_{\wedge^kE}(\xi)$ if and only if 
$\bar{\partial}^l \widehat{\xi}-\widehat{\zeta} \wedge_E \widehat{\theta}_1=0$ for $1\le k\le 2n-1$, 
$\bar{\partial}^l \widehat{\xi}-\widehat{\zeta} \wedge_E \widehat{\theta}_1=0$ and $\bar{\partial}^l \widehat{\zeta}=0$ for $k=2n$. $\hfill\Box$
\end{prop}

\section{Quaternionic $k$-vector fields} 
In this section, we consider a section of $\wedge^k E\otimes S^mH$ in the case $m=k$. 
Such a section is a $k$-vector field on $M$ since $\wedge^k E\otimes S^kH$ is the subbundle of $\wedge^k TM\otimes \mathbb{C}$. 
We call a quaternionic section of $\wedge^kE \otimes S^kH$ as a quaternionic $k$-vector field on $M$. 
We prove that any quaternionic $k$-vector field corresponds to a holomorphic $(k,0)$-vector field on $Z$. 
\subsection{Definition of quaternionic $k$-vector fields}
\begin{defi}\label{s5.1d1}
{\rm 
A section $X$ of $\wedge^k E\otimes S^kH$ is called a \textit{quaternionic $k$-vector field} on $M$ 
if $\mathfrak{D}_{\wedge^k E}^0(X)=0$ for $1\le k\le 2n-1$ and $\mathfrak{D}_{\wedge^{k-1} E}\circ \tr\circ \mathfrak{D}_{\wedge^{k} E}(X)=0$ for $k=2n$. 
If a quaternionic $k$-vector field $X$ is $\tau$-invariant, then $X$ is said to be a \textit{quaternionic real $k$-vector field} on $M$. 
}
\end{defi}
This definition is also valid in quaternionic manifolds. 
We denote by $\mathcal{Q}(\wedge^kE\otimes S^kH)$ the sheaf of quaternionic $k$-vector fields on $M$. 
The sheaf $\mathcal{Q}(\wedge^kE\otimes S^kH)^{\tau}$ of $\tau$-invariant elements of $\mathcal{Q}(\wedge^kE\otimes S^kH)$ is the sheaf of quaternionic real $k$-vector fields on $M$. 

\begin{prop}
A 1-vector field $X$ on $M$ is quaternionic if and only if $X$ preserves the quaternionic structure $Q$, 
that is, $L_XQ\subset Q$ where $L_X$ is the Lie derivative with respect to $X$. 
\end{prop}
\begin{proof}
Let $\nabla$ be the Levi-Civita connection on $M$. 
We define an endomorphism $[\nabla X, I]$ for $I\in Q$ as $\nabla X \circ I-I \circ \nabla X$. 
Since $L_XI=\nabla_X I-[\nabla X, I]$, 
the condition $L_XQ\subset Q$ is equal to $[\nabla X, Q]\subset Q$. 
Under the decomposition $\textrm{End}(TM)=(E\otimes E^*)_0 \otimes Q + \id_E\otimes Q+ E\otimes E^*\otimes \id_H$, 
$[\nabla X, Q]\subset Q$ is equivalent that $\nabla X$ has no component of $(E\otimes E^*)_0 \otimes Q$. 
This condition is written by $\mathfrak{D}_{E}^0(X)=0$ 
since $(E\otimes E^*)_0 \otimes Q\cong (E\otimes E^*)_0 \otimes S^2H$ by $\omega_{H^*}^{\sharp}$. 
\end{proof}

\subsection{Horizontal lift of $k$-vector fields to $P(H^*)$} 
Let $\widetilde{\mathcal{A}}_0(\wedge^k \widetilde{\mathcal{H}})$ denote the sheaf of $\textrm{GL}(1,\mathbb{H})$-invariant horizontal $k$-vector fields on $P(H^*)$. 
The isomorphism $\mathcal{A}^0(\wedge^k TM)\cong \widetilde{\mathcal{A}}_0(\wedge^k \widetilde{\mathcal{H}})$ is given 
by taking the horizontal lift $\widetilde{X}_h$ of a $k$-vector field $X\in \mathcal{A}^0(\wedge^k TM)$. 
From now on, we denote by $\wedge^k E S^kH$ the tensor space $\wedge^k E \otimes S^kH$ for simplicity. 
We define $\widetilde{\wedge^k E S^kH}$ as a subbundle of $\wedge^k \widetilde{\mathcal{H}}$ corresponding to $\wedge^k E S^kH$. 
Let $\widetilde{\mathcal{A}}_0(\widetilde{\wedge^k E S^kH})$ be a subsheaf of $\widetilde{\mathcal{A}}_0(\wedge^k \widetilde{\mathcal{H}})$ 
which consists of the horizontal lift of elements of $\mathcal{A}^0(\wedge^k E S^kH)$. 
Then $\mathcal{A}^0(\wedge^k E S^kH)\cong \widetilde{\mathcal{A}}_0(\widetilde{\wedge^k E S^kH})$ by $X\mapsto \widetilde{X}_h$. 
The $k$-th wedge $\widetilde{\theta}^k_0$ provides a bundle map $\widetilde{\theta}^k_0: \wedge^k \widetilde{\mathcal{H}}\to p^{-1}(\wedge^kE)$ and an isomorphism 
\begin{equation}\label{s5.4eq1}
\widetilde{\mathcal{A}}_0(\widetilde{\wedge^k E S^kH})\cong \widetilde{\mathcal{A}}^0_{(k,0)}(\wedge^k E)
\end{equation} 
such that $\widetilde{\theta}^k_0(\widetilde{X}_h)=\widetilde{X}_0$ for $X\in \mathcal{A}^0(\wedge^k E S^kH)$.

\begin{lem}\label{s5.4l1} 
Let $\widetilde{X}_h$ be an element of $\widetilde{\mathcal{A}}_0(\wedge^k E S^kH)$. 
The $(k,0)$-part $\widetilde{X}_h^{k,0}$ is $\textrm{GL}(1,\mathbb{C})$-invariant and holomorphic along each fiber of $p:P(H^*)\to M$. 
\end{lem}
\begin{proof} 
The horizontal $k$-vector field $\widetilde{X}_h$ is $\textrm{GL}(1,\mathbb{C})$-invariant. 
Since the action of $\textrm{GL}(1,\mathbb{C})$ on $P(H^*)$ is holomorphic, 
the $(k,0)$-part $\widetilde{X}_h^{k,0}$ is also $\textrm{GL}(1,\mathbb{C})$-invariant. 
The vector field $\widetilde{X}_h^{k,0}$ is holomorphic along each fiber if and only if 
$\nabla_{v}^{0,1}\widetilde{X}_h^{k,0}$ vanishes for any $v\in \mathcal{A}^0(\widetilde{\mathcal{V}})$. 
We remark that the bundle $\widetilde{\mathcal{H}}^{1,0}$ is not a holomorphic subbundle of $T^{1,0}P(H^*)$. 
The condition $\nabla_{v}^{0,1}\widetilde{X}_h^{k,0}=0$ is equal that $\nabla^{0,1}_v \widetilde{X}_h^{k,0}$ is annihilated by $\widetilde{\theta}_0^k$, $\widetilde{\theta}_0^{k-1}\wedge \eta_0$, 
$\widetilde{\theta}_0^{k-1}\wedge \eta_1$ and $\widetilde{\theta}_0^{k-2}\wedge \eta_0\wedge \eta_1$. 
By Proposition~\ref{s3.4p2}, $\widetilde{\theta}_0^{k-1}\wedge \eta_0 (\nabla^{0,1}_v \widetilde{X}_h^{k,0})$, 
$\widetilde{\theta}_0^{k-1}\wedge \eta_1 (\nabla^{0,1}_v \widetilde{X}_h^{k,0})$ and $\widetilde{\theta}_0^{k-2}\wedge \eta_0\wedge \eta_1 (\nabla^{0,1}_v \widetilde{X}_h^{k,0})$ vanish. 
Furthermore, $\widetilde{\theta}_0^k(\nabla^{0,1}_v \widetilde{X}_h^{k,0})=\bar{\partial}_v(\widetilde{\theta}_0^k(\widetilde{X}_h^{k,0}))
=\bar{\partial}_v(\widetilde{\theta}_0^k(\widetilde{X}_h))=\widetilde{\theta}_0^k(\nabla_{v}^{0,1}\widetilde{X}_h)=0$ 
for any $v\in \mathcal{A}^0(\widetilde{\mathcal{V}})$. 
Hence $\widetilde{X}_h^{k,0}$ is holomorphic along each fiber. 
\end{proof}

We denote by $\widetilde{\mathcal{A}}_0(\wedge^k \widetilde{\mathcal{H}}^{1,0})$ 
the sheaf of horizontal $(k,0)$-vector fields which are $\textrm{GL}(1,\mathbb{C})$-invariant and holomorphic along each fiber of $p:P(H^*)\to M$. 
\begin{prop}\label{s5.4p1} 
The isomorphism $\mathcal{A}^0(\wedge^k E S^kH)\cong \widetilde{\mathcal{A}}_0(\wedge^k \widetilde{\mathcal{H}}^{1,0})$ is given by $X\mapsto \widetilde{X}_h^{k,0}$. 
Moreover, $\widetilde{X}_0=\widetilde{\theta}^k_0(\widetilde{X}_h^{k,0})$ for $X\in \mathcal{A}^0(\wedge^k E S^kH)$. 
\end{prop}
\begin{proof} 
Let $X'$ be a horizontal $(k,0)$-vector field on $P(H^*)$. 
As in the proof of Lemma~\ref{s5.4l1}, $\widetilde{\theta}_0^k(\nabla^{0,1}_v X')=\bar{\partial}_v(\widetilde{\theta}_0^k(X'))$ for any $v\in \mathcal{A}^0(\widetilde{\mathcal{V}})$. 
Moreover, 
\begin{equation}\label{s5.4eq2} 
R_c^*(\widetilde{\theta}_0^k(X'))=(R_c^*\widetilde{\theta}_0^k)((R_{c^{-1}})_*X')=c^k\widetilde{\theta}_0^k((R_{c^{-1}})_*X')
\end{equation}
for any $c\in \textrm{GL}(1,\mathbb{C})$. 
Thus the bundle isomorphism $\widetilde{\theta}_0^k: \wedge^k \widetilde{\mathcal{H}}^{1,0}\cong p^{-1}(\wedge^k E)$ induces an isomorphism 
\begin{equation}\label{s5.4eq3} 
\widetilde{\mathcal{A}}_0(\wedge^k \widetilde{\mathcal{H}}^{1,0})\cong \widetilde{\mathcal{A}}^0_{(k,0)}(\wedge^k E).
\end{equation} 
By Lemma~\ref{s5.4l1}, we have a map $\widetilde{\mathcal{A}}_0(\widetilde{\wedge^k E S^kH})\to \widetilde{\mathcal{A}}_0(\wedge^k \widetilde{\mathcal{H}}^{1,0})$. 
This map is the composition of two isomorphisms in (\ref{s5.4eq1}) and (\ref{s5.4eq3}), and it is isomorphic. 
Hence $\mathcal{A}^0(\wedge^k E S^kH)\cong \widetilde{\mathcal{A}}_0(\wedge^k \widetilde{\mathcal{H}}^{1,0})$. 
It follows from the isomorphism (\ref{s5.4eq1}) that $\widetilde{X}_0=\widetilde{\theta}^k_0(\widetilde{X}_h)=\widetilde{\theta}^k_0(\widetilde{X}_h^{k,0})$ for $X\in \mathcal{A}^0(\wedge^k E S^kH)$. 
\end{proof}
Under the irreducible decomposition of $\wedge^k TM$, the horizontal lift of the components except for $\wedge^k E S^kH$ vanish by $\widetilde{\theta}^k_0$. 
Hence, Proposition~\ref{s5.4p1} implies the following : 
\begin{cor}\label{s5.4c1} 
Let $X$ be an element of $\mathcal{A}^0(\wedge^k TM)$. 
The $(k,0)$-part $\widetilde{X}_h^{k,0}$ of the horizontal lift $\widetilde{X}_h$ is $\textrm{GL}(1,\mathbb{C})$-invariant and holomorphic along each fiber of $p:P(H^*)\to M$. 
$\hfill\Box$
\end{cor}

\subsection{Holomorphic lift of quaternionic $k$-vector fields to $P(H^*)$} 
A horizontal $(k,0)$-vector field $X'$ on $P(H^*)$ is called of \textit{ $\textrm{GL}(1,\mathbb{C})$-order $m$ }
if $(R_{c^{-1}})_*X'=c^mX'$ for any $c\in \textrm{GL}(1,\mathbb{C})$. 
We define $\widetilde{\mathcal{A}}_m(\wedge^k \widetilde{\mathcal{H}}^{1,0})$ as 
the sheaf of horizontal $(k,0)$-vector fields which are of $\textrm{GL}(1,\mathbb{C})$-order $m$ and holomorphic along each fiber of $p:P(H^*)\to M$. 
By the equation (\ref{s5.4eq2}), we obtain an isomorphism 
\begin{equation}\label{s5.5eq1}
\widetilde{\mathcal{A}}_m(\wedge^k \widetilde{\mathcal{H}}^{1,0})\cong \widetilde{\mathcal{A}}^0_{(k+m,0)}(\wedge^k E)
\end{equation}
given by $X'\mapsto \widetilde{\theta}_0^k(X')$. 
Let $\xi$ be an element of $\mathcal{A}^0(\wedge^k E S^{k+m}H)$. 
The lift $\widetilde{\xi}_0$ of $\xi$ to $P(H^*)$ is in $\widetilde{\mathcal{A}}^0_{(k+m,0)}(\wedge^k E)$. 
By the isomorphism (\ref{s5.5eq1}), there exists a unique element $\widetilde{Y}_{\xi}$ of $\widetilde{\mathcal{A}}_m(\wedge^k \widetilde{\mathcal{H}}^{1,0})$ 
such that 
\[
\widetilde{\theta}^k_0(\widetilde{Y}_{\xi})=\widetilde{\xi}_0.
\] 
Using the isomorphism (\ref{s5.5eq1}), we have 
\begin{equation}\label{s5.5eq2}
\mathcal{A}^0(\wedge^k E S^{k+m}H)\cong \widetilde{\mathcal{A}}_m(\wedge^k \widetilde{\mathcal{H}}^{1,0})
\end{equation} 
by $\xi\mapsto \widetilde{Y}_{\xi}$. 
In the case $m=0$, by considering $\xi$ as the $k$-vector field $X$ on $M$, 
the isomorphism $\mathcal{A}^0(\wedge^k E S^kH)\cong \widetilde{\mathcal{A}}_0(\wedge^k \widetilde{\mathcal{H}}^{1,0})$ 
is given by $X\mapsto \widetilde{X}_h^{k,0}$.

\begin{prop}\label{s5.5p2} 
Let $X$ and $\zeta$ be elements of $\mathcal{A}^0(\wedge^k E S^kH)$ and $\mathcal{A}^0(\wedge^{k-1} E S^{k+1}H)$, respectively. 
The $k$-vector field $X$ is quaternionic and $\zeta=\tr\circ \mathfrak{D}_{\wedge^kE}(X)$ if and only if 
there exist $Y_0\in \mathcal{A}^0_{P(H^*)}(\wedge^{k-1} \widetilde{\mathcal{H}}^{1,0})$ and $Z_0\in \mathcal{A}^0_{P(H^*)}(\wedge^{k-2} \widetilde{\mathcal{H}}^{1,0})$ such that 
the $(k,0)$-vector field $\widetilde{X}_h^{k,0}+Y_0\wedge v_0+Y_1\wedge v_1+ Z_0\wedge v_0\wedge v_1$ is holomorphic 
for $Y_1=\frac{1}{k^2r^2}\widetilde{Y}_{\zeta}$. 
\end{prop}
\begin{proof} 
Proposition~\ref{s5.4p1} implies that $\widetilde{X}_0=\widetilde{\theta}_0^k(\widetilde{X}_h^{k,0})$, 
and $\widetilde{\zeta}_0=\widetilde{\theta}_0^{k-1}(\widetilde{Y}_{\zeta})$. 
Setting $Y_1=\frac{1}{k^2r^2}\widetilde{Y}_{\zeta}$, then we obtain $\widetilde{\zeta}_0=k^2\widetilde{\theta}_0^{k-1}(r^2Y_1)$. 
It follows from Proposition~\ref{s4.5p2} that $X$ is quaternionic and $\zeta=\tr\circ \mathfrak{D}_{\wedge^kE}(X)$ if and only if 
$\bar{\partial}\widetilde{X}_0-\widetilde{\zeta}_0\wedge_E r^{-2}\widetilde{\theta}_1=0$ for $1\le k\le 2n-1$, 
$\bar{\partial}\widetilde{X}_0-\widetilde{\zeta}_0\wedge_E r^{-2}\widetilde{\theta}_1=0$ 
and $\bar{\partial}\widetilde{\zeta}_0=0$ for $k=2n$. 
The condition is equivalent to 
$\bar{\partial}(\widetilde{\theta}_0^k(\widetilde{X}_h^{k,0}))-k^2\widetilde{\theta}_0^{k-1}(Y_1)\wedge_E \widetilde{\theta}_1=0$ for $1\le k\le 2n-1$, 
$\bar{\partial}(\widetilde{\theta}_0^k(\widetilde{X}_h^{k,0}))-k^2\widetilde{\theta}_0^{k-1}(Y_1)\wedge_E \widetilde{\theta}_1=0$ 
and $\bar{\partial}(\widetilde{\theta}_0^{k-1}(r^2Y_1))=0$ for $k=2n$. 
It is equal that there exist $Y_0\in \mathcal{A}^0_{P(H^*)}(\wedge^{k-1} \widetilde{\mathcal{H}}^{1,0})$, $Z_0\in \mathcal{A}^0_{P(H^*)}(\wedge^{k-2} \widetilde{\mathcal{H}}^{1,0})$ 
such that $\widetilde{X}_h^{k,0}+Y_0\wedge v_0+Y_1\wedge v_1+ Z_0\wedge v_0\wedge v_1$ is holomorphic by Theorem~\ref{s5.3t1}. 
\end{proof}

\subsection{Horizontal lift of $k$-vector fields to $Z$} 
We denote by $\widehat{\mathcal{A}}_0(\wedge^k \widehat{\mathcal{H}})$ the sheaf of horizontal smooth $k$-vector fields which are constant along each fiber of $f:Z\to M$. 
For $X\in \mathcal{A}^0(\wedge^kTM)$, the horizontal lift $\widehat{X}_h$ is an element of $\widehat{\mathcal{A}}_0(\wedge^k \widehat{\mathcal{H}})$. 
Conversely, $\widehat{\mathcal{A}}_0(\wedge^k \widehat{\mathcal{H}})$ consists of such elements. 
Hence, we obtain an isomorphism $\mathcal{A}^0(\wedge^kTM)\cong \widehat{\mathcal{A}}_0(\wedge^k \widehat{\mathcal{H}})$ by $X\mapsto \widehat{X}_h$. 
We denote by $\widehat{\mathcal{A}}(\wedge^k \widehat{\mathcal{H}}^{1,0})$ 
the sheaf of horizontal $(k,0)$-vector fields which are holomorphic along each fiber of $f$. 
The vector field $\widehat{X}_h$ and the $(k,0)$-part $\widehat{X}_h^{k,0}$ correspond to $\widetilde{X}_h$ and $\widetilde{X}_h^{k,0}$, respectively. 
Proposition~\ref{s5.4p1} induces the following : 
\begin{prop}\label{s5.7p1} 
The isomorphism $\mathcal{A}^0(\wedge^k E S^kH)\cong \widehat{\mathcal{A}}(\wedge^k \widehat{\mathcal{H}}^{1,0})$ is given by $X\mapsto \widehat{X}_h^{k,0}$. 
Moreover, $\widehat{X}=\widehat{\theta}^k_0(\widehat{X}_h^{k,0})$ for $X\in \mathcal{A}^0(\wedge^k E S^kH)$. $\hfill\Box$
\end{prop}

Corollary~\ref{s5.4c1} implies 
\begin{cor}\label{s5.7c1} 
Let $X$ be an element of $\mathcal{A}^0(\wedge^k TM)$. 
The $(k,0)$-part $\widehat{X}_h^{k,0}$ of the horizontal lift $\widehat{X}_h$ is holomorphic along each fiber of $f:Z\to M$. 
$\hfill\Box$
\end{cor}

We consider the holomorphic bundle $\wedge^k \widehat{\mathcal{H}}^{1,0}\otimes l^m$ for a non-negative integer $m$. 
Let $\widehat{\mathcal{A}}(\wedge^k \widehat{\mathcal{H}}^{1,0}\otimes l^m)$ be 
a sheaf of $l^m$-valued horizontal smooth $(k,0)$-vector fields which are holomorphic along each fiber of $f:Z\to M$. 
Let $\widehat{\mathcal{O}}(\wedge^k \widehat{\mathcal{H}}^{1,0}\otimes l^m)$ denote the subsheaf of $\widehat{\mathcal{A}}(\wedge^k \widehat{\mathcal{H}}^{1,0}\otimes l^m)$ 
of holomorphic $l^m$-valued horizontal $(k,0)$-vector fields. 
By the definition of $l$, we obtain the isomorphism 
\begin{equation}\label{s5.7eq1}
\widehat{\mathcal{A}}(\wedge^k \widehat{\mathcal{H}}^{1,0}\otimes l^m)\cong \widetilde{\mathcal{A}}_m(\wedge^k \widetilde{\mathcal{H}}^{1,0}). 
\end{equation}
The isomorphism $\widehat{\theta}_0^k: \wedge^k \widehat{\mathcal{H}}^{1,0} \to f^{-1}(\wedge^k E)\otimes l^k$ is extended to 
$\widehat{\theta}_0^k\otimes (\id_{l})^m : \wedge^k \widehat{\mathcal{H}}^{1,0}\otimes l^m \to f^{-1}(\wedge^k E)\otimes l^{k+m}$. 
\begin{lem}\label{s5.7l1} 
The map $\widehat{\theta}_0^k\otimes (\id_{l})^m$ induces 
the isomorphisms $\widehat{\mathcal{A}}(\wedge^k \widehat{\mathcal{H}}^{1,0}\otimes l^m)\cong \widehat{\mathcal{A}}^0(\wedge^k E\otimes l^{k+m})$ 
and $\widehat{\mathcal{O}}(\wedge^k \widehat{\mathcal{H}}^{1,0}\otimes l^m)\cong \widehat{\mathcal{O}}(\wedge^k E\otimes l^{k+m})$. 
\end{lem}
\begin{proof} 
Let $X''$ be an $l^m$-valued horizontal $(k,0)$-vector field. 
The covariant derivative $\nabla^{0,1} X''$ is a section of $\wedge^k\widehat{\mathcal{H}}^{1,0}\otimes l \otimes \wedge^{0,1}T^*Z $ 
since $\widehat{\mathcal{H}}^{1,0}$ is a holomorphic subbundle of $T^{1,0}Z$. 
It follows from Proposition~\ref{s3.7p4} that 
$\bar{\partial}^l(\widehat{\theta}_0^k\otimes (\id_{l})^m(X''))=\widehat{\theta}_0^k\otimes (\id_{l})^m(\nabla^{0,1} X'')$. 
Thus $\bar{\partial}^l(\widehat{\theta}_0^k\otimes (\id_{l})^m(X''))=0$ if and only if $\nabla^{0,1}X''=0$. 
Hence, we obtain the two isomorphisms in this proposition. 
\end{proof}
We denote $\widehat{\theta}_0^k\otimes (\id_{l})^m$ by $\widehat{\theta}_0^k$ for short. 
Let $\xi$ be an element of $\mathcal{A}^0(\wedge^k E S^{k+m}H)$. 
The lift $\widehat{\xi}$ of $\xi$ to $Z$ is in $\widehat{\mathcal{A}}^0(\wedge^k E\otimes l^{k+m})$. 
Then there exists a unique element $\widehat{Y}_{\xi}$ of $\widehat{\mathcal{A}}(\wedge^k \widehat{\mathcal{H}}^{1,0}\otimes l^m)$ 
such that 
\begin{equation}\label{s5.7eq2}
\widehat{\theta}^k_0(\widehat{Y}_{\xi})=\widehat{\xi}
\end{equation}
by Lemma~\ref{s5.7l1}. 
The isomorphisms in (\ref{s5.5eq2}) and (\ref{s5.7eq1}) yield 
\begin{equation}\label{s5.7eq3}
\mathcal{A}^0(\wedge^k E S^{k+m}H)\cong \widehat{\mathcal{A}}(\wedge^k \widehat{\mathcal{H}}^{1,0}\otimes l^m)
\end{equation} 
by $\xi\mapsto \widehat{Y}_{\xi}$. 
In the case $m=0$, by considering $\xi$ as a $k$-vector field $X$ in $M$, 
the isomorphism $\mathcal{A}^0(\wedge^k E S^kH)\cong \widehat{\mathcal{A}}(\wedge^k \widehat{\mathcal{H}}^{1,0})$ 
is given by $X\mapsto \widehat{X}_h^{k,0}$. 
We consider the operator 
\begin{equation*}
\mathfrak{D}_{\wedge^kE}: \mathcal{A}^0(\wedge^k E S^{k+m}H) \to \mathcal{A}^0(\wedge^k E\otimes E^*\otimes S^{k+m+1}H). 
\end{equation*}
It follows from the isomorphism (\ref{s4.6e3}) that $\Ker \mathfrak{D}_{\wedge^kE} \cong \widehat{\mathcal{O}}(\wedge^k E\otimes l^{k+m})$ by $\xi\mapsto \widehat{\xi}$. 
By Lemma~\ref{s5.7l1}, we obtain the following isomorphism : 
\begin{cor}\label{s5.7c1} 
$\Ker \mathfrak{D}_{\wedge^kE} \cong \widehat{\mathcal{O}}(\wedge^k \widehat{\mathcal{H}}^{1,0}\otimes l^m)$ by $\xi\mapsto \widehat{Y}_{\xi}$. 
$\hfill\Box$
\end{cor}

\subsection{Holomorphic lift of quaternionic $k$-vector fields to $Z$} 
\begin{prop}\label{s5.8p1} 
Let $X$ and $\zeta$ be elements of $\mathcal{A}^0(\wedge^k E S^kH)$ and $\mathcal{A}^0(\wedge^{k-1} E S^{k+1}H)$, respectively. 
The $k$-vector field $X$ is quaternionic and $\zeta=\tr\circ \mathfrak{D}_{\wedge^kE}(X)$ if and only if 
the $(k,0)$-vector field $\widehat{X}_h^{k,0}+Y\wedge v$ is holomorphic for $Y=k^{-2}\widehat{Y}_{\zeta}$. 
\end{prop}
\begin{proof} 
It follows from Proposition~\ref{s5.7p1} that $\widehat{X}=\widehat{\theta}_0^k(\widehat{X}_h^{k,0})$. 
We set $Y=k^{-2}\widehat{Y}_{\zeta}$. Then $\widehat{\zeta}=k^2\widehat{\theta}_0^{k-1}(Y)$ by (\ref{s5.7eq2}). 
Proposition~\ref{s4.7p2} implies that $X$ is quaternionic and $\zeta=\tr\circ \mathfrak{D}_{\wedge^kE}(X)$ if and only if 
$\bar{\partial}^l \widehat{X}-\widehat{\zeta} \wedge_E \widehat{\theta}_1=0$ for $1\le k\le 2n-1$, 
$\bar{\partial}^l \widehat{X}-\widehat{\zeta} \wedge_E \widehat{\theta}_1=0$ and $\bar{\partial}^l \widehat{\zeta}=0$ for $k=2n$. 
They are written by $\bar{\partial}^l (\widehat{\theta}_0^k(\widehat{X}_h^{k,0}))-k^2\widehat{\theta}_0^{k-1}(Y) \wedge_E \widehat{\theta}_1=0$ for $1\le k\le 2n-1$, 
$\bar{\partial}^l (\widehat{\theta}_0^k(\widehat{X}_h^{k,0}))-k^2\widehat{\theta}_0^{k-1}(Y) \wedge_E \widehat{\theta}_1=0$ and $\bar{\partial}^l (\widehat{\theta}_0^{k-1}(Y))=0$ for $k=2n$. 
The condition is equal that $\widehat{X}_h^{k,0}+Y\wedge v$ is holomorphic for any $k$ by Theorem~\ref{s5.6t1}. 
\end{proof}
Let $\widehat{\mathcal{O}}(\wedge^k T^{1,0}Z)$ be a sheaf of holomorphic $(k,0)$-vector fields defined in the pull-back of open sets on $M$ by $f:Z\to M$. 
\begin{thm}\label{s5.8t1} 
An isomorphism $\mathcal{Q}(\wedge^kE S^kH)\cong \widehat{\mathcal{O}}(\wedge^k T^{1,0}Z)$ is given by 
$X\mapsto \widehat{X}_h^{k,0}+Y\wedge v$ where $Y$ is defined by $k^{-2}\widehat{Y}_{\tr\circ \mathfrak{D}_{\wedge^kE}(X)}$. 
Moreover, $H^0(\mathcal{Q}(\wedge^kE S^kH))$ is isomorphic to the space $H^0(\mathcal{O}(\wedge^k T^{1,0}Z))$ of holomorphic $k$-vector fields on $Z$ by the correspondence. 
\end{thm}
\begin{proof} 
By Proposition~\ref{s5.8p1}, we obtain a map $\mathcal{Q}(\wedge^kE S^kH)\to \widehat{\mathcal{O}}(\wedge^k T^{1,0}Z)$ by 
$X\mapsto \widehat{X}_h^{k,0}+Y\wedge v$. 
Conversely, any element $X'$ of $\widehat{\mathcal{O}}(\wedge^k T^{1,0}Z)$ is written by 
$X'=X'_h+Y\wedge v$ for $X'_h\in \widehat{\mathcal{A}}^0(\wedge^k \widehat{\mathcal{H}}^{1,0})$, $Y\in \widehat{\mathcal{A}}(\wedge^{k-1} \widehat{\mathcal{H}}^{1,0}\otimes l^2)$. 
The isomorphism (\ref{s5.7eq3}) implies that there exist $X\in \mathcal{A}^0(\wedge^k E S^kH)$ and $\zeta\in \mathcal{A}^0(\wedge^{k-1} E S^{k+1}H)$ 
such that $X'_h=\widehat{X}_h^{k,0}$ and $Y=k^{-2}\widehat{Y}_{\zeta}$. 
Since $X'=\widehat{X}_h^{k,0}+Y\wedge v$ is holomorphic, 
$X$ is quaternionic and $\zeta=\tr\circ \mathfrak{D}_{\wedge^kE}(X)$ by Proposition~\ref{s5.8p1}. 
Thus we obtain the isomorphism $\mathcal{Q}(\wedge^kE S^kH)\cong \widehat{\mathcal{O}}(\wedge^k T^{1,0}Z)$ by $X\mapsto \widehat{X}_h^{k,0}+Y\wedge v$. 
If $X$ is a global section of $\mathcal{Q}(\wedge^kE S^kH)$, then $\zeta=\tr\circ \mathfrak{D}_{\wedge^kE}(X)$ and $Y=k^{-2}\widehat{Y}_{\zeta}$ are also global. 
Hence $H^0(\mathcal{Q}(\wedge^kE S^kH))$ is isomorphic to $H^0(\widehat{\mathcal{O}}(\wedge^k T^{1,0}Z))=H^0(\mathcal{O}(\wedge^k T^{1,0}Z))$ by $X\mapsto \widehat{X}_h^{k,0}+Y\wedge v$. 
It completes the proof. 
\end{proof}

\subsection{Holomorphic lift of quaternionic real $k$-vector fields to $Z$}
We denote by $\widetilde{\mathcal{A}}^0_{P(H^*)}(\wedge^k TP(H^*))$ the sheaf $p^{-1}p_*\mathcal{A}^0_{P(H^*)}(\wedge^k TP(H^*))$ on $P(H^*)$. 
An endomorphism $\widetilde{\tau}$ of $\widetilde{\mathcal{A}}^0_{P(H^*)}(\wedge^k TP(H^*))$ is defined as 
\[
\widetilde{\tau}(X')=\overline{(R_{-j})_*X'}
\] 
for $X'\in \widetilde{\mathcal{A}}^0_{P(H^*)}(\wedge^k TP(H^*))$. 
The action of $j$ preserves the horizontal space $\widetilde{\mathcal{H}}$ and it is anti-holomorphic. 
Thus $\widetilde{\tau}$ induces endomorphisms of $\widetilde{\mathcal{A}}_0(\wedge^k \widetilde{\mathcal{H}})$ and $\widetilde{\mathcal{A}}_0(\wedge^k \widetilde{\mathcal{H}}^{1,0})$. 
In fact, $\widetilde{\tau}$ is the complex conjugate on $\widetilde{\mathcal{A}}_0(\wedge^k \widetilde{\mathcal{H}})$. 
The endomorphism $\tau=J_E\otimes J_H$ on $TM$ is extended to $\wedge^k TM$, and it is also the complex conjugate on $\mathcal{A}^0(\wedge^kTM)$. 
Thus $\mathcal{A}^0(\wedge^kTM)^{\tau}\cong \widetilde{\mathcal{A}}_0(\wedge^k \widetilde{\mathcal{H}})^{\widetilde{\tau}}$ is given by 
taking the horizontal lift of real $k$-vector fields of $M$. 
Since $\widetilde{\tau}$ preserves the horizontal $(k,0)$-vector space $\wedge^k \widetilde{\mathcal{H}}^{1,0}$, 
$\widetilde{\mathcal{A}}_0(\wedge^k \widetilde{\mathcal{H}})^{\widetilde{\tau}}\to \widetilde{\mathcal{A}}_0(\wedge^k \widetilde{\mathcal{H}}^{1,0})^{\widetilde{\tau}}$ 
by $\widetilde{X}_h\mapsto \widetilde{X}_h^{k,0}$. 
Hence we have the isomorphism $\mathcal{A}^0(\wedge^kE S^kH)^{\tau}\cong \widetilde{\mathcal{A}}_0(\wedge^k \widetilde{\mathcal{H}}^{1,0})^{\widetilde{\tau}}$ 
by $X\mapsto \widetilde{X}_h^{k,0}$. 

The map $\widetilde{\tau}$ induces an endomorphism of $\widetilde{\mathcal{A}}_m(\wedge^k \widetilde{\mathcal{H}}^{1,0})$ 
since 
\[
(R_{c^{-1}})_*\overline{(R_{-j})_*X'}=\overline{(R_{-j})_*(R_{\bar{c}^{-1}})_*X'}=\overline{(R_{-j})_*\bar{c}^mX'}=c^m\overline{(R_{-j})_*X'}
\]
for $X'\in \widetilde{\mathcal{A}}_m(\wedge^k \widetilde{\mathcal{H}}^{1,0})$. 
Moreover, 
\[
\widetilde{\tau}(\widetilde{\theta}_0^k(X'))=J_E\overline{R_j^*(\widetilde{\theta}_0^k(X'))}
=J_E\overline{(R_j^*\widetilde{\theta}_0^k)}(\overline{(R_{-j})_*X'))}=\widetilde{\tau}(\widetilde{\theta}_0^k)(\widetilde{\tau}(X'))
=\widetilde{\theta}_0^k(\widetilde{\tau}(X'))
\] 
for $X'\in \widetilde{\mathcal{A}}_m(\wedge^k \widetilde{\mathcal{H}}^{1,0})$. 
It yields that $\widetilde{\mathcal{A}}_m(\wedge^k \widetilde{\mathcal{H}}^{1,0})^{\widetilde{\tau}}\cong \widetilde{\mathcal{A}}^0_{(k+m,0)}(\wedge^k E)^{\widetilde{\tau}}$ 
by $X'\mapsto \widetilde{\theta}_0^k(X')$. 
By Corollary~\ref{s3.1.5c1}, we have an $\mathbb{R}$-isomorphism 
\begin{equation}\label{s5.9eq1}
\mathcal{A}^0(\wedge^k E S^{k+m}H)^{\tau}\cong \widetilde{\mathcal{A}}_m(\wedge^k \widetilde{\mathcal{H}}^{1,0})^{\widetilde{\tau}}
\end{equation} 
by $\xi\mapsto \widetilde{Y}_{\xi}$. 
In the case $m=0$, the isomorphism $\mathcal{A}^0(\wedge^k E S^kH)^{\tau}\cong \widetilde{\mathcal{A}}_0(\wedge^k \widetilde{\mathcal{H}}^{1,0})^{\widetilde{\tau}}$ 
is given by $X\mapsto \widetilde{X}_h^{k,0}$. 

We denote by $\widehat{\mathcal{A}}^0_Z(\wedge^k TZ)$ the sheaf $f^{-1}f_*\mathcal{A}^0_Z(\wedge^k TZ)$ on $Z$. 
An endomorphism $\widehat{\tau}$ of $\widehat{\mathcal{A}}^0_Z(\wedge^k TZ)$ is defined by 
\[
\widehat{\tau}(X')=\overline{(R_{[j]})_*X'}
\] 
for $X'\in \widehat{\mathcal{A}}^0_Z(\wedge^k TZ)$. 
Then $\widehat{\tau}$ induces endomorphisms of 
$\widehat{\mathcal{A}}_0(\wedge^k \widehat{\mathcal{H}})$ and $\widehat{\mathcal{A}}(\wedge^k \widehat{\mathcal{H}}^{1,0})$. 
Under isomorphisms $\widehat{\mathcal{A}}_0(\wedge^k \widehat{\mathcal{H}})\cong \widetilde{\mathcal{A}}_0(\wedge^k \widetilde{\mathcal{H}})$ 
and $\widehat{\mathcal{A}}(\wedge^k \widehat{\mathcal{H}}^{1,0}) \cong \widetilde{\mathcal{A}}_0(\wedge^k \widetilde{\mathcal{H}}^{1,0})$, 
the elements $\widehat{\tau}(\widehat{X}_h), \widehat{\tau}(\widehat{X}_h^{k,0})$ correspond to $\widetilde{\tau}(\widetilde{X}_h), \widetilde{\tau}(\widetilde{X}_h^{k,0})$ 
for $X\in \mathcal{A}^0(\wedge^kE S^kH)$, respectively. 
Therefore, $\mathcal{A}^0(\wedge^kE S^kH)^{\tau}\cong \widehat{\mathcal{A}}(\wedge^k \widehat{\mathcal{H}}^{1,0})^{\widehat{\tau}}$ 
by $X\mapsto \widehat{X}_h^{k,0}$. 
Moreover, $\widehat{\tau}$ induces a bundle map of $l^m$, and it is extended to an endomorphism of $\widehat{\mathcal{A}}(\wedge^k \widehat{\mathcal{H}}^{1,0} \otimes l^m)$. 
Then $\widetilde{\mathcal{A}}_m(\wedge^k \widetilde{\mathcal{H}}^{1,0})^{\widetilde{\tau}}\cong \widehat{\mathcal{A}}(\wedge^k \widehat{\mathcal{H}}^{1,0}\otimes l^m)^{\widehat{\tau}}$. 
Using the isomorphism (\ref{s5.9eq1}), we obtain an $\mathbb{R}$-isomorphism 
\begin{equation}\label{s5.9eq2}
\mathcal{A}^0(\wedge^k E S^{k+m}H)^{\tau}\cong \widehat{\mathcal{A}}(\wedge^k \widehat{\mathcal{H}}^{1,0}\otimes l^m)^{\widehat{\tau}}
\end{equation} 
by $\xi\mapsto \widehat{Y}_{\xi}$. 
In the case $m=0$, the isomorphism $\mathcal{A}^0(\wedge^k E S^kH)^{\tau}\cong \widehat{\mathcal{A}}(\wedge^k \widehat{\mathcal{H}}^{1,0})^{\widehat{\tau}}$ 
is given by $X\mapsto \widehat{X}_h^{k,0}$. 
Moreover, 
\[
(\Ker \mathfrak{D}_{\wedge^kE})^{\tau} \cong \widehat{\mathcal{O}}(\wedge^k \widehat{\mathcal{H}}^{1,0}\otimes l^m)^{\widehat{\tau}}
\] 
under the correspondence. 
Corresponding to Proposition~\ref{s5.8p1}, we obtain the following :
\begin{prop}\label{s5.9p3} 
Let $X$ and $\zeta$ be elements of $\mathcal{A}^0(\wedge^k E S^kH)$ and $\mathcal{A}^0(\wedge^{k-1} E S^{k+1}H)$, respectively. 
The $k$-vector field $X$ is quaternionic and real, and $\zeta=\tr\circ \mathfrak{D}_{\wedge^kE}(X)$ if and only if 
the $(k,0)$-vector field $\widehat{X}_h^{k,0}+Y\wedge v$ is holomorphic and $\widehat{\tau}$-invariant for $Y=k^{-2}\widehat{Y}_{\zeta}$. 
\end{prop}
\begin{proof} 
Let $X$ be a quaternionic $k$-vector field on $M$ and $\zeta$ the element $\tr\circ \mathfrak{D}_{\wedge^kE}(X)$. 
It suffices to show that $X$ is real if and only if $\widehat{X}_h^{k,0}+Y\wedge v$ is $\widehat{\tau}$-invariant for $Y=k^{-2}\widehat{Y}_{\zeta}$. 
If $X$ is real, then $\zeta=\tr\circ \mathfrak{D}_{\wedge^kE}(X)$ is $\tau$-invariant. 
By the isomorphism (\ref{s5.9eq2}), $\widehat{X}_h^{k,0}$ and $Y=k^{-2}\widehat{Y}_{\zeta}$ are $\widehat{\tau}$-invariant. 
Now we have $\widehat{\tau}(v)=v$ since $\widetilde{\tau}(v_1)=v_1$. 
Thus $\widehat{X}_h^{k,0}+Y\wedge v$ is also $\widehat{\tau}$-invariant. 
Conversely, we assume that $\widehat{X}_h^{k,0}+Y\wedge v$ is $\widehat{\tau}$-invariant. 
Then $\widehat{\tau}(\widehat{X}_h^{k,0})=\widehat{X}_h^{k,0}$ 
since $\widehat{\tau}$ preserves the decomposition $\wedge^k T^{1,0}Z=\wedge^k \widehat{\mathcal{H}}^{1,0}\oplus (\wedge^{k-1} \widehat{\mathcal{H}}^{1,0})\wedge \widehat{\mathcal{V}}^{1,0}$. 
It follows from (\ref{s5.9eq2}) that $X$ is $\tau$-invariant, that is, real. 
Hence we finish the proof. 
\end{proof}

Proposition~\ref{s5.9p3} implies  
\begin{thm}\label{s5.9t1} 
An $\mathbb{R}$-isomorphism $\mathcal{Q}(\wedge^kE S^kH)^{\tau}\cong \widehat{\mathcal{O}}(\wedge^k T^{1,0}Z)^{\widehat{\tau}}$ is given by 
$X\mapsto \widehat{X}_h^{k,0}+Y\wedge v$ where $Y$ is defined by $k^{-2}\widehat{Y}_{\tr\circ \mathfrak{D}_{\wedge^kE}(X)}$. 
Moreover, $H^0(\mathcal{Q}(\wedge^kE S^kH))^{\tau}\cong H^0(\mathcal{O}(\wedge^k T^{1,0}Z))^{\widehat{\tau}}$ by the correspondence. 
$\hfill\Box$
\end{thm}

\subsection{Example}
Let $M$ be the $n$-dimensional quaternionic projective space $\mathbb{H}P^n$. 
Then $P(H^*)=\mathbb{C}^{2n+2}\backslash \{0\}$ as a complex manifold. 
The twistor space $Z$ is the $(2n+1)$-dimensional complex projective space $\mathbb{C}P^{2n+1}$. 
We consider the space $H^0(\mathcal{Q}(\wedge^kE S^kH))$ of quaternionic $k$-vector fields on $\mathbb{H}P^n$. 
We denote by $(z_0, z_1,\dots, z_{2n+1})$ the standard coordinate of $\mathbb{C}^{2n+2}$. 
The vector field $v_0$ associated with the action of ${\rm GL}(1, \mathbb{C})$ on $P(H^*)$ is written by 
\[
v_0=z_0\frac{\partial}{\partial z_0}+z_1\frac{\partial}{\partial z_1}+\dots+z_{2n+1}\frac{\partial}{\partial z_{2n+1}} 
\]
on $\mathbb{C}^{2n+2}\backslash\{0\}$. 
Let $\widetilde{V}_k$ denote the space of ${\rm GL}(1, \mathbb{C})$-invariant holomorphic $k$-vector fields on $\mathbb{C}^{2n+2}\backslash \{0\}$. 
Then 
\[
\widetilde{V}_k=\Bigl\{\sum a_{i_1 \cdots i_k j_1\cdots j_k}z_{i_1}\cdots z_{i_k}\frac{\partial}{\partial z_{j_1}}\wedge \cdots \wedge \frac{\partial}{\partial z_{j_k}} \Bigm| a_{ijkl}\in \mathbb{C}\Bigr\} 
\]
We regard the coefficient $(a_{i_1 \cdots i_k j_1\cdots j_k})$ as an element of $\otimes^k\mathbb{C}^{2n+2}\otimes \otimes^k(\mathbb{C}^{2n+2})^*$. 
Then $(a_{i_1 \cdots i_k j_1\cdots j_k})$ is in $S^k\mathbb{C}^{2n+2}\otimes \wedge^k(\mathbb{C}^{2n+2})^*$. 
We define $S^k\otimes \wedge^k$ as the projection from $\otimes^k gl(2n+2, \mathbb{C})\cong \otimes^k\mathbb{C}^{2n+2}\otimes \otimes^k(\mathbb{C}^{2n+2})^*$ 
to $S^k\mathbb{C}^{2n+2}\otimes \wedge^k(\mathbb{C}^{2n+2})^*$. 
Then $\widetilde{V}_k\cong S^k\otimes \wedge^k(\otimes^k gl(2n+2, \mathbb{C}))$. 
The space $H^0(\mathcal{O}(\wedge^k T^{1,0}Z))$ of holomorphic $k$-vector fields on $\mathbb{C}P^{2n+1}$ is identified with 
the quotient space $\widetilde{V}_k/\widetilde{V}_{k-1}\wedge v_0$, 
where $\widetilde{V}_{k-1}\wedge v_0$ is the subspace of $\widetilde{V}_k$ consisting of 
$a_{i_1 \cdots i_{k-1} j_1\cdots j_{k-1}}z_{i_1}\cdots z_{i_{k-1}}\frac{\partial}{\partial z_{j_1}}\wedge \cdots \wedge \frac{\partial}{\partial z_{j_{k-1}}}\wedge v_0$ (c.f. \S 5.1 in \cite{MN2}): 
\[
H^0(\mathcal{O}(\wedge^k T^{1,0}Z))\cong \widetilde{V}_k/\widetilde{V}_{k-1}\wedge v_0
\]
The space $\widetilde{V}_{k-1}\wedge v_0$ is isomorphic to $S^k\otimes \wedge^k(\otimes^{k-1} gl(2n+2, \mathbb{C})\otimes \textrm{Id})$. 
Theorem~\ref{s5.8t1} implies that 
\[
H^0(\mathcal{Q}(\wedge^kE S^kH))\cong S^k\otimes \wedge^k(\otimes^k gl(2n+2, \mathbb{C}))/ S^k\otimes \wedge^k(\otimes^{k-1} gl(2n+2, \mathbb{C})\otimes \textrm{Id}) 
\]
and  
\[
\dim_{\mathbb{C}} H^0(\mathcal{Q}(\wedge^kE S^kH))= \sum_{i=0}^k(-1)^{k+i}{}_{2n+i+1}C_i\ {}_{2n+2}C_i. 
\]
The real structure $\tau$ is associated with the action of $j$ on $\mathbb{H}^{n+1}\backslash \{0\}\cong \mathbb{C}^{2n+2}\backslash \{0\}$. 
Since $\widetilde{V}_k^{\widehat{\tau}}\cong S^k\otimes \wedge^k(\otimes^k gl(n+1, \mathbb{H}))$, 
\[
H^0(\mathcal{Q}(\wedge^kE S^kH))^{\tau}\cong S^k\otimes \wedge^k(\otimes^k gl(n+1, \mathbb{H}))/ S^k\otimes \wedge^k(\otimes^{k-1} gl(n+1, \mathbb{H})\otimes \textrm{Id}) 
\]
by Theorem~\ref{s5.9t1}. 

Especially, in the case $\mathbb{H}P^1\cong S^4$, 
\begin{align*}
H^0(\mathcal{Q}(E H))^{\tau}&\cong gl(2, \mathbb{H})/ \textrm{Id} \cong sl(2, \mathbb{H}), \\
H^0(\mathcal{Q}(\wedge^2E S^2H))^{\tau}&\cong S^2\otimes \wedge^2(\otimes^2 gl(2, \mathbb{H}))/ S^2\otimes \wedge^2(gl(2, \mathbb{H})\otimes \textrm{Id}) 
\end{align*}
and $\dim_{\mathbb{R}} H^0(\mathcal{Q}(E H))^{\tau}= 15$, $\dim_{\mathbb{R}} H^0(\mathcal{Q}(\wedge^2 E S^2H))^{\tau}= 45$. 

\vspace{\baselineskip}
\noindent
\textbf{Acknowledgements}. 
The first named author is supported by Grant-in-Aid for Young Scientists (B) $\sharp$17K14187 from JSPS.



\end{document}